\documentclass{article}

\usepackage{amssymb}
\usepackage{amsfonts}
\usepackage{amsmath}
\usepackage{enumerate}
\usepackage{graphicx}
\usepackage{tikz}
\usepackage{hyperref}
\usepackage[all]{xy}
\usepackage{xcolor}

\hypersetup{ colorlinks   = true,   urlcolor     = blue,   linkcolor    = blue,   citecolor   = red }

\newtheorem{theorem}{Theorem}[section]
\newtheorem{lemma}[theorem]{Lemma}
\newtheorem{corollary}[theorem]{Corollary}
\newtheorem{proposition}[theorem]{Proposition}

\newtheorem{notation}[theorem]{Notation}
\newtheorem{assumption}[theorem]{Assumption}
\newtheorem{algorithm}[theorem]{Algorithm}
\newtheorem{definition}[theorem]{Definition}
\newtheorem{example}[theorem]{Example}

\newtheorem{remark}[theorem]{Remark}
\newenvironment{proof}{\paragraph{Proof:}}{\hfill$\square$}
\numberwithin{equation}{section}

\begin{document}
\title{Modular companions in planar one-dimensional equisymmetric strata}
\author{S. Allen Broughton, Antonio F. Costa\thanks{Partially supported by PID2023-152822 NB-I00 of Ministerio de Ciencia, Innovacion y Universidades (Spain)}, Milagros Izquierdo}\maketitle

\begin{abstract}
Consider, in the moduli space of Riemann surfaces of a fixed genus, the subset of surfaces with non-trivial automorphisms. 
Of special interest are the numerous subsets of surfaces admitting an action of a given finite group, $G$, acting with a specific signature. 
In a previous study \cite{BrCoIz2}, we declared two Riemann surfaces to be  \emph{modular companions} if they have topologically equivalent $G$ actions, 
and that their $G$ quotients are conformally equivalent orbifolds.  
In this article we present a geometrically-inspired measure to
decide whether two modular companions are conformally equivalent (or how different), respecting the $G$ action.  
Along the way, we construct a moduli space for surfaces with the specified $G$ action and associated equivariant tilings on these surfaces. 
We specifically apply the ideas to planar, finite group actions whose quotient orbifold is a sphere with four cone points.
\end{abstract}


\section{Overview, motivation, and background}
\subsection{Overview and motivation}\label{subsec-overmotiv}

In this paper we have two goals: 

\medskip
\paragraph{Goal 1.}  \emph{Moduli of group actions}.  Put the authors' previous study \cite{BrCoIz1}, \cite{BrCoIz2} of group actions 
and equisymmetric strata  into a unified framework, using the ideas of \emph{horizontal and vertical invariants} (Section \ref{subsec-horivert}). 

\paragraph{Goal 2.} \emph{Modular companions} (Section \ref{subsec-modcomp}) are an essential part of the uniform framework of Goal 1. 
Our second goal is to distinguish modular companions by geometric means, rather than using the algebraic nature of their definition.  
As an intermediate step, we construct equivariant tilings on surfaces and their Cayley graphs.

\medskip
The main contribution of the paper is Goal 2, though it needs Goal 1 to put the narrative into a proper context.  Our examples focus primarily 
on signatures defining planar, one dimensional strata (Section \ref{subsubsec-eqsymm} and Remark \ref{rk-stratafacts}), but the methods apply to higher dimensional strata. 

The \emph{group action moduli spaces} that we construct have been considered by other authors \cite{GDH,HPCRH}. 
Our contribution is to realize these spaces as finite orbifold covers of the moduli space of quotients, 
given in equations \eqref{eq-pBM}, \eqref{eq-modcomp-comp}, and \eqref{eq-modcomp-map}. 
The key idea is the horizontal and vertical invariants of an action, see Section \ref{subsec-horivert} for more detail. 

\subsubsection{Motivation}\label{subsubsec-motiv}
Let $S$ be  a Riemann surface of genus $\sigma \ge 2$. The collection of all conformal equivalence classes of such surfaces 
is called the \emph{moduli space of surfaces of genus $\sigma$} which we denote by $\mathcal{M}_\sigma$. 
The moduli space is both a quasi-projective algebraic variety and a complex orbifold: 
\begin{equation*}
  \mathcal{M}_\sigma=\mathcal{T}_\sigma/M_\sigma.
\end{equation*}
We shall explain the right hand side (orbifold), along with the terminology to follow, a bit later in the paper.

The singularity set $\mathcal{B}_\sigma$ of $\mathcal{M}_\sigma$, considered as a complex orbifold, is called the branch locus, 
an object of much study, see \cite{BaCoIz, Br1, BrCoIz1, BrCoIz2, Ha, HPCRH} for instance.
The branch locus consists of all surfaces with a ``larger than expected'' automorphism group, or with ``extra automorphisms''. 
 
\vskip 1pt
\noindent
Note: When $\sigma = 2$ the expected automorphism group has 2 elements, consisting of the identity and the hyperelliptic involution.
In higher genus the only expected automorphism is the identity.

The branch locus has a stratification into \emph{equisymmetric strata}, 
which consist of surfaces with ``the same automorphism group'', see Sections \ref{subsubsec-eqsymm}, \ref{subsec-eqstrata}. 
The strata are irreducible, smooth, quasi-projective subvarieties of $\mathcal{M}_\sigma$. 
The closures of these strata can be modeled by the components (also called strata) of a 
moduli space of surfaces allowing a conformal action of a finite group $G$ 
with a specific \emph{action signature} $\mathfrak{s}$.  These \emph{group action moduli spaces}, 
introduced in Section \ref{sec-modaction}, are a principal focus of this paper.  
Using these spaces, the study of $\mathcal{B}_\sigma$  amounts to a detailed 
study of the different types of automorphism groups of surfaces of a fixed genus.

Throughout the paper, all our moduli spaces and maps will be viewed as complex analytic spaces and maps,
specifically complex orbifolds and covering maps. This is also the viewpoint of the authors \cite{GDH,HPCRH}.
This is enough to do all relevant topology and, additionally, highlights the so-called\emph{ branch loci} as a natural orbifold construction.
We will not try to force things into the algebro-geometric category, which can obscure the branch locus in low genus. 
However, without defining equations, fields of definition and other such arithmetic questions are out of bounds.

\subsubsection{Examples and computation}\label{subsubsec-examples}
Throughout the paper we illustrate the ideas by a series of examples, 
\ref{ex-S3-2233-1}, \ref{ex-S3-2233-2}, \ref{ex-S3-2233-3}, \ref{ex-S3-2233-4}, discussing the small case of the non-abelian group $G=\Sigma_3$, 
acting with signature $\mathfrak{s}=(0;2,2,3,3)$ upon a surface $S$ of genus $2$. This small case is amenable to hand calculation. 
If the group or  genus is of moderate size then computer methods must be used, especially for computing partial isometries. 
For instance, suppose $G=\mathrm{Alt}_5$, acting with signature $\mathfrak{s}=(0;5,5,5,5)$ on a surface $S$. 
Then the genus of  $S$ is $37$ and there are $47$ conformally inequivalent surfaces $S$ with a  
$(G, \mathfrak{s})$ action and such that for each $S$, $S/G$ is the same quotient surface,
namely a sphere with $4$ specific cone points of order $5$. For a discussion of more complex examples see 
Section \ref{subsec-spectacular}.

Detailed Magma \cite{Mag} scripts and sample calculations are given on the site \cite{Br3}, Section 4. 
Code for all the calculations, including modular action on generating vector classes, 
computing tiling invariants, and partial isometry calculations, are posted. 
Complete calculation log files for some interesting groups and signatures are also posted.       

\subsubsection{Outline}
In the remainder of this paper, we shall focus on this sequential list of topics:
\begin{enumerate}
  \item Section \ref{subsec-background}: recall necessary background for the rest of the paper: orbifolds, equisymmetry, conformal group actions, uniformization,  and monodromy. 
  \item Section \ref{sec-modaction}: uniform construction of group action moduli spaces and their relation to the equisymmetric stratification. 
  \item Section \ref{sec-tiling}: equivariant tilings of surfaces.  
  \item Section \ref{sec-S4}: detailed analysis of equivariant tilings for one dimensional planar actions --  four cone points on a sphere.  
  \item Section \ref{sec-Cayley}: dual and group Cayley graphs for an equivariant tiling and  
                 partial isometries of modular companions. 
  \item Section \ref{sec-BasicLemmas}: basic lemmas used in other sections. 
                 
\end{enumerate}

\subsection{Acknowledgement}
The authors wish to thank the referee for numerous suggestions about clarifying the exposition, especially encouraging us to add a final subsection on interesting and ``spectacular" examples.

\subsection{Background}\label{subsec-background}
Now, let us describe the ideas of orbifolds, equisymmetry, and group actions, in the remainder of this introductory section.
\subsubsection{Orbifolds} 
We utilize orbifolds that may be realized as a quotient of a manifold by a discrete group (good orbifolds).  For more background on orbifolds, see \cite{MM}.

Suppose we have a real or complex manifold $X$ and  a discrete group $\mathcal{G}$ acting effectively and discontinuously, 
as a group of transformations of $X$, via $\left(g,x\right)  \rightarrow g\cdot x=gx$. 
We call $(X,\mathcal{G})$ an orbifold pair, and denote the orbits by $\mathcal{G}x=x^{\mathcal{G}}$, for $x \in X$; the space of orbits by $X/\mathcal{G}$; 
and the quotient map $x\rightarrow\mathcal{G}x=x^{\mathcal{G}}$ by $\pi_\mathcal{G} : X\rightarrow X/\mathcal{G}$.
For the topological properties of the pair $(X,\mathcal{G})$ and the orbifold $X/\mathcal{G}$, we are specifically assuming that: 
\begin{enumerate}
  \item The manifold $X$ has countably many path components (by definition), $\mathcal{G}$  is countable and the $\mathcal{G}$-stabilizer of every $x \in  X$  is finite.
  \item For any two points $x \in X$, $y \in X -\mathcal{G}x$, there exists a pair of neighborhoods 
         $U_x$, $U_y$ of $x$, respectively $y$, such that for $ \forall g \in  G$, $gU_x \cap U_y = \varnothing$.
  \item The topology on $X/\mathcal{G}$ is induced by the quotient map  $\pi_\mathcal{G}$. 
  \item For each $x \in X$ there is a neighbourhood  $U_x$ of $X$ such that $U_x$ does not meet $gU_x$ unless $gx = x$, and $U_x$ is invariant under the stabilizer $\mathcal{G}_x=\{g\in\mathcal{G}: gx=x\}$.
  \item For $U_x$ so chosen, $U_x/\mathcal{G}_x \rightarrow \pi_\mathcal{G}(U_x)$ is a homeomorphism onto a neighbourhood of $y=\pi_{\mathcal{G}}(x)$ in $X/\mathcal{G}$.
\end{enumerate}
Items 4 and 5 are consequences of items 1,2, and 3. 

\medskip
Suppose we are given two orbifold pairs  $(X,\mathcal{G})$ and  $(Y,\mathcal{H})$ and a continuous map $\phi:X\rightarrow Y$ that satisfies
\[\phi(gx)= w(g)\phi(x),\]
for some map $w:\mathcal{G}\rightarrow \mathcal{H}$. Then $\phi$ induces a  map $\overline{\phi}:  X/\mathcal{G} \rightarrow Y/\mathcal{H}$, as we summarize in the 
diagram below.
\begin{equation}\label{dia-orbifold-lift}
\xymatrix{
   X \ar[r]^\phi \ar[d]^{\pi_{\mathcal{G}}} & Y \ar[d]^{\pi_{\mathcal{H}}} \\
  X/\mathcal{G} \ar[r]^{\overline{\phi}}       & Y/\mathcal{H}
   }
\end{equation}
We call the map $\phi$ a \emph{lift} of  $\overline{\phi}$. 

Now suppose that $\overline{\phi}: X/\mathcal{G}\rightarrow Y/\mathcal{H}$ is a map of orbifolds.  We say that 
$\overline{\phi}$ is an \emph{orbifold cover} if for each $\overline{y} \in Y/\mathcal{H}$, there is an open neigbourhood
$\overline{V}$ of $\overline{y}$ such that $(\overline{\phi})^{-1}(\overline{V})$ is a disjoint union of open sets $\{\overline{U}_i\}$,  
each of which contains a unique point $\overline{x}_i \in \overline{U}_i \cap (\overline{\phi})^{-1}(\overline{y})$. 
Moreover, we require that the restriction of $\overline{\phi}: (\overline{U}_i,\overline{x}_i) \rightarrow (\overline{V},\overline{y})$ has a lift 
$\phi: (U,x_i) \rightarrow (V,y)$ as in diagram \eqref{dia-orbifold-lift},
where the vertical actions are by stabilizers of points $x_i$ and $y$. The neighbourhood $\overline{V}$ has to be chosen
so that $V$ and all the lifted neighbourhoods $U_i$ are invariant under the stabilizers of $y$ and $x_i$. 

We are not going to  define the \emph{orbifold fundamental group} $\pi_1^{orb}(X/\mathcal{G}, \overline{x}_0)$  here, 
but do note that if $X$ is simply connected, then $\pi_1^{orb}(X/\mathcal{G}, \overline{x}_0)$
is isomorphic to $\mathcal{G}$ as a group of deck transformations of the map $X \rightarrow X/\mathcal{G}$. Later, in
Section \ref{subsubsec-monolift}, we show how we may use standard covers to construct lifts and monodromies.

The next proposition gives a universal way to construct finite orbifold covers of a good orbifold.
It generalizes the orbifold cover defined by a subgroup of finite index $\mathcal{K} \subseteq \pi_1^{orb}(X)$, through an action on the  
coset space $\pi_1^{orb}(X)/\mathcal{K}$. The proof is straightforward and we leave it to the reader. 

\begin{proposition}\label{prop-orbicover}
Suppose that $(X, \mathcal{G})$ is an orbifold pair, $X$ is connected, and $\mathcal{F}$ is a finite set upon which $\mathcal{G}$ acts via 
$\left(g,v\right)  \rightarrow g\cdot v=gv$.
 Let $\mathcal{G}$ act upon $X\times\mathcal{F}$ by the diagonal action  $g\cdot
(x,v)=(g\cdot x,g\cdot v)$, and consider the projection
\begin{equation*}
p:\left(  X\times\mathcal{F}\right)  /\mathcal{G}\rightarrow X/\mathcal{G}%
,\text{ }\left(  x,v\right)  ^{\mathcal{G}}\rightarrow x^{\mathcal{G}}.
\end{equation*}
Then, we have the following. 
\begin{enumerate}
  \item The map  $p$ is well-defined and is an orbifold cover.
  \item Each path component of $\left(  X\times\mathcal{F}\right)  /\mathcal{G}$,
  is determined by an orbit $\mathcal{F}_i = \mathcal{G}\cdot f_i$ and we have a decomposition into disjoint orbifolds
  \[  
   \left( X\times\mathcal{F} \right) /\mathcal{G}=\bigcup_i \left( X\times\mathcal{F}_i \right)  /\mathcal{G}, 
  \]
  with connected orbifold covering projections
  \[
     p_i =  p:\left(  X\times\mathcal{F}_i\right)  /\mathcal{G}\rightarrow X/\mathcal{G}.
  \]
  \item For each $(x,v) \in (x_0,v_0)^\mathcal{G}$ the map $p$ has the localized form 
  \[ 
     U\times \{v\}/(\mathcal{G}_x\cap\mathcal{G}_v) \rightarrow  U/\mathcal{G}_x,
  \]  
  where $U$ is a suitable $\mathcal{G}_x$ invariant neighbourhood of $x$.
  \item For each $x\in X$, the fibres $p^{-1}(p(x))$ are in 1-1 correspondence to the orbits of the stabilizer $\mathcal{G}_{x}$ acting upon $\mathcal{F}$.
\end{enumerate}
\end{proposition}

\subsubsection{Equisymmetry}\label{subsubsec-eqsymm}   Fix a genus $\sigma \geq 2$, and let $S_0$ be a reference surface and $S$ an arbitrary surface, both of genus $\sigma$. 
Next, let $h:S_0\rightarrow S$ be any orientation preserving homeomorphism. The automorphism group  $\textrm{Aut}(S)$  
is finite and we may pull it back to a finite subgroup of orientation preserving homeomorphisms of $S_0$, 
via the map $\phi \rightarrow h^{-1}\circ \phi \circ h$. The pullback of the automorphism group then defines
a finite subgroup $F \subset M_\sigma \simeq MCG(S_0)$, the mapping class group of $S_0$. It is well-known that, for $\sigma \geq 2$,
the map $\textrm{Aut}(S) \rightarrow F$ is an isomorphism. If a different homeomorphism $h^\prime$ is chosen,
then the corresponding subgroup $F^\prime$ is conjugate to $F$. Thus $S$ determines a conjugacy class of finite subgroups of $M_\sigma$, 
 called the\emph{ symmetry type} of $S$ and it is denoted by $\Sigma(S)$.

Two Riemann surfaces $S_1$ and $S_2$ are called \emph{equisymmetric} if $\Sigma(S_1)=\Sigma(S_2)$. 
The equivalence classes of this relation are the so-called  \emph{equisymmetric strata}, referred to in Section \ref{subsubsec-motiv}. 
As mentioned, the strata are irreducible, smooth, quasi-projective subvarieties of $\mathcal{M}_\sigma$. 
The branch locus is a disjoint union  of equisymmetric strata, and the open subvariety of generic surfaces 
is $\mathcal{M}_\sigma-\mathcal{B}_\sigma$.   
The closure of an equisymmetric stratum may be modeled by one of the components of a \emph{group action moduli space}, the focus of our paper.
On the other hand, an equisymmetric stratum closure is obtained adding the equisymmetric strata of ``higher symmetry''. 
For more background on the equisymmetric stratification see \cite{Br1}.    

\subsubsection{Group actions and signatures}
A conformal action of the finite group $G$ upon the closed, connected Riemann surface $S$ is a monomorphism 
$\epsilon:G \rightarrow \mathrm{Aut}(S)$, where $\mathrm{Aut}(S)$ is the group of conformal automorphisms of $S$. 
If the action is twisted by $\omega \in \mathrm{G}$, say, $\epsilon^{\prime}=\epsilon\circ\omega^{-1}$,  
the same \emph{action subgroup} $\epsilon^\prime(G)=\epsilon(G)$ of $\mathrm{Aut}(S)$ is determined.\
The pair $(S,\epsilon)$ is a called an \emph{action pair} and $(S,\epsilon(G))$ is called a \emph{subgroup pair}.
The subgroup pair depends only on the  $\mathrm{Aut}(G)$ class of an action $\epsilon$.  
We use the finer construct of action pairs instead of subgroup pairs 
to simplify our later translation to group theory calculations. 

Let $T=S/G=S/\epsilon(G)$, as an orbifold quotient, $T$ depends only on the $\mathrm{Aut}(G)$ class of an action $\epsilon$. 
The quotient $T$ is a compact surface of genus $h$ with $r$ hyperbolic cone points $z_1,\ldots,z_r$,
and the quotient map $\pi_G: S \rightarrow T$ is branched exactly over the cone points. 
In local coordinates, over the cone point $z_j$, the map $\pi_G$ is given by $z\rightarrow z^{m_j}$, 
for some $m_j \ge 2$. The  $m_j$ must be the orders of non-trivial elements of $G$. The $(r+1)$-tuple
\begin{equation}\label{eq-signature}
\mathfrak{s}=(h;m_1,\ldots,m_r)
\end{equation}
is called the \emph{signature of the action}. 
The ordering or labelling of the cone points is important, though the $m_j$ are typically given in non-decreasing order.

\begin{remark}\label{rk-stratafacts}
We collect here some useful notes about actions, signatures, and strata.
\begin{enumerate}
  \item Following \cite{ZVC} we call an action \emph{planar} if the quotient genus $h$ is zero. 
  In this case,  for convenience,  we may abbreviate the signature to $(m_1,\ldots,m_r)$.
  \item The relation between the  genus, the group order, and the signature is given by the Riemann-Hurwitz equation
    \begin{equation}\label{eq-RH} 
      \frac{2\sigma-2}{|G|} = 2h-2+r-\sum_{j=1}^r\frac{1}{m_j}.
\end{equation}  
  \item 
  As previously announced, for each possible group-signature pair $(G,\mathfrak{s})$, 
  there is a group actions moduli space, which we will define and discuss in detail in Section \ref{sec-modaction}. 
  If non-empty, this action space maps to the union of the closures one or more equisymmetric strata, all of the 
  same complex orbifold dimension. In the non-empty case, the common dimension of the strata closures 
   is called the \emph{Teichm\"{u}ller dimension} of the action space and the corresponding strata.
  The Teichm\"{u}ller dimension equals $3h-3 +r $  when the genus $\sigma$ of the surfaces $S$ is 2 or greater. 
  We may take this number as the definition of the \emph{Teichm\"{u}ller dimension of the action pair $(G,\mathfrak{s})$},
  even if the action space is empty.
\end{enumerate}
\end{remark}

\subsubsection{Uniformization}\label{subusbsec-uniformization} 

The $G$ action may be uniformized by Fuchsian groups as follows. There is pair of Fuchsian groups
\begin{equation}\label{eq-Fuchspair}
\Pi \trianglelefteq \Gamma < \mathrm{Aut}(\mathbb{H}),
\end{equation} 
acting on the hyperbolic plane $\mathbb{H}$,  
such that $\Pi$ is torsion free, $S\simeq \mathbb{H}/\Pi$, $T\simeq \mathbb{H}/\Gamma$, and the natural homomorphism 
$\Gamma/\Pi \rightarrow \mathrm{Aut}(S)$ takes $\Gamma/\Pi$ isomorphically onto $\epsilon(G)$.  
This is conveniently summarized in the following diagram, where the arrows are quotient maps by the groups decorating the arrows.
\begin{equation}\label{dia-uniformize} 
\xymatrix{
    \mathbb{H} \ar[r]^\Pi \ar[d]^{\Gamma} & S \ar[d]^{G} \\
    T \ar[r]^{id}       & T }
\end{equation}
The signature of $\Gamma$ is the same as the signature of the $G$ action.

We will identify the Fuchsian group $\Gamma$ with the  orbifold fundamental group of $T$, for a suitable base point $z_0$:
\begin{equation} 
\Gamma \simeq \pi_{1}^{orb}(T,z_0).\label{eq-Pi1b}
\end{equation}
The group $\pi_{1}^{orb}(T,z_0)$ is a topological construction, 
but it is automatically isomorphic to $\Gamma$  as a group of orbifold deck transformations 
since $\mathbb{H}$ is simply connected and $T\simeq \mathbb{H}/\Gamma$.

\subsubsection{Monodromies and path lifting}\label{subsubsec-monolift}

We discuss monodromies next, see  \cite{Mas} for more detail on monodromy. 
By using path lifting from $T$ to $S$, via $\pi_G$, the action $\epsilon$ determines a \emph{smooth monodromy}
\begin{equation}\label{eq-monodromy}
\xi:\Gamma\simeq \pi_{1}^{orb}(T,z_0) \twoheadrightarrow G,
\end{equation}
 well defined up to inner automorphisms of  $G$.
(For completeness, we give specific details of the construction of $\xi$ by path lifting in Section \ref{subsubsub-regcovers} below.)

Recall that a \emph{monodromy} must be an epimorphism, and it is \emph{smooth} if it has a torsion-free kernel.
We can use the realization of $\Gamma$ as a cocompact Fuchsian group to show that $\xi$ is smooth.
When $\xi$ is suitably normalised, by a proper choice of lift point lying over $z_0$, the relation between $\epsilon$ and $\xi$ becomes 
\begin{equation}\label{eq-eps-xi}
  \overline{\xi}=\epsilon^{-1},
\end{equation}
where $\overline{\xi}$ is the quotient isomorphism $\Gamma/\ker(\xi)\rightarrow G$. 
In this pairing, for $\omega \in \mathrm{Aut}(G)$, we have the association 
\begin{equation}\label{eq-eps-xi2}
\omega\circ \xi \leftrightarrow \epsilon\circ\omega^{-1}.
\end{equation}

In the other direction, upon fixing the conformal orbifold structure of $T$, 
any smooth monodromy $\xi^{\prime}:\Gamma\rightarrow G$ determines an action $\epsilon^{\prime}$ of
$G$ on a possibly different surface $S^{\prime}=\mathbb{H}/\ker(\xi^{\prime})$.
We use equation \eqref{eq-eps-xi} and the natural action of $\Gamma/\ker(\xi^{\prime})$ 
on $\mathbb{H}/\ker(\xi^{\prime})$ to define the action $\epsilon^\prime$. 
The orbit spaces satisfy the following conformal equivalence relation.
\begin{equation*}
S^{\prime}/\epsilon^{\prime}(G)=\mathbb{H}/\Gamma=S/\epsilon(G). 
\end{equation*}
We observe that $S^\prime \rightarrow T$ is an orbifold cover of finite degree of the compact space $T$, so $S^\prime$ is a closed Riemann surface. 
If $\xi^{\prime}=\omega\circ\xi$ for some $\omega\in \mathrm{Aut}(G)$, 
then $\ker(\xi^{\prime})=\ker(\xi)$, and so it follows that
$S^{\prime}$ is conformally equivalent to $S$, indeed we may take them to be the same:
\begin{equation*}
S^{\prime}=\mathbb{H}/\ker(\xi^{\prime})=\mathbb{H}/\ker(\xi)=S.
\end{equation*}

We call  $(T,\xi)$ a \emph{(smooth) monodromy pair}. 
The prior discussion, including equations \eqref{eq-monodromy} to \eqref{eq-eps-xi2}, 
show that we have a 1-1 identification  of $\rm{Aut}(G)$ classes of action pairs $(S,\epsilon)$ 
and $\rm{Aut}(G)$ classes of  monodromy pairs $(T,\xi)$. The discussion also shows that for classifying surfaces
with non-trivial automorphisms only the $\rm{Aut}(G)$ classes of action pairs or monodromy pairs need to be considered.

\subsubsection{Regular covering spaces and path lifting}\label{subsubsub-regcovers}
We can use regular covering space arguments to supplement our orbifold covering space methods.
Assume that $(X,G)$ is a orbifold pair with orbifold quotient $Y=X/G$ and orbifold cover $p:X \rightarrow Y$. 
Let  $p:X^\circ \rightarrow Y^\circ$ be the regular cover over the regular points, 
which we assume is a connected covering space. For each $g \in G$ we let $\epsilon(g)$ 
denote the deck transformation $x\rightarrow gx$. Pick a base point $y_0 \in Y^\circ$, 
and let $x_0$ be a specific point lying over $y_0$. We define a monodromy 
$\xi_{x_0}:\Gamma^\circ=\pi_1(Y^\circ,y_0) \rightarrow G$ as follows. 
For $\gamma \in \Gamma^\circ$, let $\tilde{\gamma}$ be a lift to $X^\circ$ with  
$\tilde{\gamma}(0)=x_0$. Since the cover is regular, there is a $g$ such that
\begin{equation}\label{eq-lift1}
\tilde{\gamma}(1)=\epsilon(g)(x_0)=gx_0,
\end{equation}
and then we set
\begin{equation}\label{eq-lift2}
\xi_{x_0}(\gamma)=g.
\end{equation}
 
This map is surjective since the covering space is connected and $\epsilon$ is a isomorphism to the deck transformation group.
If $hx_0$ is any other point lying over $y_0$, then
\begin{equation}\label{eq-conjlift}
\xi_{hx_0}(\gamma)=h\xi_{x_0}(\gamma)h^{-1}, \text{\ } \gamma \in \Gamma^\circ.
\end{equation} 
If we start off with a given monodromy $\xi:\Gamma^\circ \rightarrow G$, we can construct 
a covering space $p:X^\circ \rightarrow Y^\circ$ and find  $x_0 \in X^\circ$,
so that $\xi=\xi_{x_0}$. Thus, in our path lifting computations, we can do all our work using 
the loops on $Y^\circ$, a given monodromy $\xi$, and the three equations immediately above.

\subsection{Horizontal and vertical invariants}\label{subsec-horivert}
It turns out that the monodromy pairs are computationally more convenient for studying families of surfaces with automorphisms.
Let us formalize the relationship between the two types of pairs, discussed in the previous section.
Every action pair $(S,\epsilon)$ of $G$ on a surface $S$ with
signature $\mathfrak{s}$, uniquely determines and is uniquely determined by two things:
\begin{enumerate}
\item A continuous invariant (horizontal): The orbifold of $T=S/G$.  
\item A discrete or combinatorial invariant (vertical): The smooth monodromy $\xi:\Gamma\rightarrow G$, determined by $\epsilon$,
see equation \eqref{eq-eps-xi}.
\end{enumerate}
Specifically, an action pair $(S,\epsilon)$ produces the two unique items in $1, 2$, hence a monodromy pair $(T,\xi)$. Furthermore, each combination $(T,\xi)$ of the items in $1, 2$ produces a unique action $(S,\epsilon)$.
The two different invariants are combined in the map \eqref{eq-pBM} which we introduce a bit later. The base and fibres of the map \eqref{eq-pBM} are the items $1, 2$, which is why  we use the terms \emph{horizontal and vertical invariants}. 
The association $(S,\epsilon)\leftrightarrow (T,\xi)$ preserves the $\mathrm{Aut}(G)$ orbits on $G$ actions and monodromies. 

We will need to compare conformal equivalence classes of action pairs and monodromy pairs.
Two action pairs $(S_1, \epsilon_1)$, $(S_2, \epsilon_2)$ are \emph{conformally equivalent} if there is a conformal isomorphism 
$h: S_1 \rightarrow S_2$ and $\omega \in \mathrm{Aut}(G)$ such that 
\begin{equation}\label{eq-equiv-action}
(\epsilon_2\circ \omega^{-1})(g) =h\circ\epsilon_1(g)\circ h^{-1},\text{  } \forall g \in G.
\end{equation}
If action pairs are conformally equivalent then the signatures of the two $G$ actions are the same, upon suitable 
labelling of the cone points in both quotients. The following characterization is worth stating as a proposition, 
and may be found in \cite{GDH} and \cite{HPCRH}.
\begin{proposition}\label{prop-equivactions}
Let notation be as above. If $S_1$ and $S_2$ are both equal to a common surface $S$, 
then the image groups $\epsilon_1(G)$ and $\epsilon_2(G)$ are conjugate in $\mathrm{Aut}(S)$ if and only if the actions are conformally equivalent. 
\end{proposition}

Likewise, we declare that two monodromy pairs $(T_1, \xi_1)$, $(T_2, \xi_2)$ are conformally equivalent if there is a conformal isomorphism 
$k: T_1 \rightarrow T_2$, respecting the cone point structure, and  $\omega \in \mathrm{Aut}(G)$ such that 
\begin{equation}\label{eq-equiv-mono}
(\omega\circ\xi_2)(\gamma) = (\xi_1\circ k_\ast^{-1})(\gamma), \text{  } \forall \gamma \in \Gamma_2=\pi_1^{orb}(T_2,z_0),
\end{equation}
where $k_\ast$ is the induced map on fundamental groups.
Our definitions take into account that the subgroups and surfaces constructed are invariant under the $\mathrm{Aut}(G)$ action.  

\begin{remark}\label{rk-topequiv}
If the map $h: S_1 \rightarrow S_2$ is a homeomorphism then the action pairs
$(S_1, \epsilon_1)$ and $(S_2, \epsilon_2)$ are called \emph{topologically equivalent}.
There is a similar statement for  $k: T_1 \rightarrow T_2$ and the monodromy pairs $(T_1, \xi_1)$ and $(T_2, \xi_2)$. 
For more information on the equivalence of surfaces with automorphisms, see \cite{Br2}.   
\end{remark}
We finish this section with an example illustrating some of the ideas we have discussed. We will continue using the example for illustration  
as we introduce additional ideas. 

\begin{example}\label{ex-S3-2233-1} 
Let $G=\Sigma_3= \langle x,y:x^2=y^3=1,y^x=y^{-1} \rangle$. Assume that $G$ acts on a surface $S$ with signature $\mathfrak{s}=(0;2,2,3,3)$.  
By the Riemann-Hurwitz formula \eqref{eq-RH}, the genus of  $S$ is 2.
The corresponding Fuchsian group $\Gamma$ has this presentation:  
\begin{equation}\label{eq-GammaPrez}
 \Gamma =\Gamma_\mathfrak{s}=\left\langle \gamma _{1},\gamma _{2},\gamma _{3},\gamma _{4}:\gamma _{1}\gamma _{2}\gamma _{3}\gamma _{4}
 =\gamma _{1}^{2}=\gamma _{2}^{2}=\gamma _{3}^{3}=\gamma _{4}^{3}=1\right\rangle.
\end{equation}
Any monodromy $\xi$ is determined by the generating vector:
\begin{equation}\label{eq-GV-S3-2233}
\mathcal{V} = (c_1,c_2,c_3,c_4) = (\xi(\gamma _{1}),\xi(\gamma _{2}),\xi(\gamma _{3}),\xi(\gamma _{4})).
\end{equation}
(See Section \ref{subsec-ActionModSpace} for a more on generating vectors.)
Since $G = \langle c_1,c_2,c_3,c_4\rangle$ and the  $\mathcal{V}$  must satisfy all the relations in \eqref{eq-GammaPrez}
then we can determine that there are  $12$ such vectors. Therefore, there are two $\mathrm{Aut}(G)$ classes of vectors.
Representatives of these classes are the vectors $\mathcal{V}_1=(x,x,y,y^{-1})$ and   $\mathcal{V}_2=(x,xy,y,y)$. 
\end{example}

\section{Group action moduli spaces}\label{sec-modaction}

\subsection{A moduli space for group actions}\label{subsec-ActionModSpace}
Let $S_{0}$ be a reference surface and select some orientation preserving
homeomorphism $h:S_{0}\rightarrow S$. We may pull back a conformal $G$-action
on $S$, via $h$, to a topological action of $h^{-1}Gh$ on $S_{0}$, and we have
an identification $\overline{h}:S_{0}/h^{-1}Gh\rightarrow S/G$. On the group
level, setting $\Gamma$ $=\pi_{1}^{orb}(S/G)$, $\Gamma_{0}$ $=\pi_{1}%
^{orb}(S_{0}/h^{-1}Gh)$ we get an isomorphism $\overline{h}_{\ast}:\Gamma
_{0}\rightarrow\Gamma$. The group $\Gamma$ is isomorphic to a Fuchsian group
$\rho(\Gamma)$ $\subset PSL_{2}(\mathbb{R})$, where the map $\rho:\Gamma$
$\rightarrow PSL_{2}(\mathbb{R})$ is a monomorphism, and such that $S/G$ is
conformally equivalent to $\mathbb{H}/\rho(\Gamma)$. We can make all these
various Fuchsian groups isomorphic images of $\Gamma_{0}$ by $\rho
\circ\overline{h}_{\ast}:\Gamma_{0}\rightarrow PSL_{2}(\mathbb{R})$. If
$h^{\prime}$ $:S_{0}\rightarrow S$ is another homeomorphism and $x\in
PSL_{2}(\mathbb{R})$, then setting $\rho^{\prime}=\mathrm{Ad}_{x}\circ
\rho\circ\overline{h^{\prime}}_{\ast}$, $\rho^{\prime}(\Gamma)$ is a subgroup
conjugate to $\rho(\Gamma)$ and the new orbifold $\mathbb{H}/\rho^{\prime
}(\Gamma)$ is conformally equivalent to $\mathbb{H}/\rho(\Gamma)$.

\begin{notation}\label{note-Adx}
The adjoint representation by inner automorphisms is defined by 
$\mathrm{Ad}_{x}(y)=xyx^{-1}=y^{x^{-1}}$, it is an  automorphism of any group in which $x$  and $y$ reside.
\end{notation}

Following the methods in \cite{MaSi}, for a fixed signature $\mathfrak{s}$, we define the Teichm\"{u}ller space
\begin{equation}\label{eq-teichdef}
\mathcal{T}_{\mathfrak{s}}=\left\{  \rho:\Gamma_{0}\rightarrow PSL_{2}(\mathbb{R}%
)\right\}  /\left(  \mathrm{conjugation}\text{ \textrm{in }}PSL_{2}%
(\mathbb{R})\right)  .
\end{equation}
Following our previous discussion, if $\mu=\left(  \overline{h^{\prime}%
}\right)_\ast  ^{-1}\circ\overline{h}_{\ast}\in$ $M_{\mathfrak{s}}=\mathrm{Out}^{+}(\Gamma
_{0})$, and $x\in PSL_{2}(\mathbb{R})$, then
\begin{equation}\label{eq-Teichmodaction}
 \rho^{\prime}=\mathrm{Ad} _{x}\circ\rho\circ\mu^{-1}.
\end{equation}
Thus $M_{\mathfrak{s}}$ acts upon $\mathcal{T}_{\mathfrak{s}}$ and the
entire orbit $M_{\mathfrak{s}}\cdot\rho$ consists of various conjugacy classes of Fuchsian groups
with conformally equivalent orbit spaces. With a little work, we can show that
distinct orbits correspond to conformally inequivalent orbit spaces. Thus, the
moduli space 
\begin{equation}\label{eq-modspacedef}
\mathcal{M}_{\mathfrak{s}}=\mathcal{T}_{\mathfrak{s}}/M_{\mathfrak{s}}
\end{equation}
solves our horizontal invariant problem of classifying conformal equivalence classes of orbifolds $S/G$.

Now suppose we have a monodromy $\xi_{\Gamma}:\Gamma\rightarrow G$ determining
a surface $S$ lying over $\mathbb{H}/\rho(\Gamma)$. As previously discussed we need only work 
with the $\mathrm{Aut}(G)$ classes of monodromies:
\begin{equation}\label{eq-AutGmono}%
\xi_{\Gamma}^{\mathrm{Aut}(G)}=\left\{  \omega\circ\xi_{\Gamma}:\omega\in\mathrm{Aut}%
(G)\right\}  . 
\end{equation}
Since $\xi_{\Gamma}$ is surjective, then $\xi_{\Gamma}^{\mathrm{Aut}(G)}$ has $\left\vert
\mathrm{Aut}(G)\right\vert $ distinct monodromies. Now, $\xi=\xi_{\Gamma}%
\circ\overline{h}_{\ast}:\Gamma_{0}\rightarrow G$, is a monodromy based on our
reference surface $S_{0}$. The association  $\xi\leftrightarrow$ $\xi_{\Gamma}$ is a
bijection, and $\xi^{\mathrm{Aut}(G)}\leftrightarrow$ $\xi_{\Gamma
}^{\mathrm{Aut}(G)}$ is a bijection between classes of monodromies. So, we will now assume
that every monodromy class  $\xi^{\mathrm{Aut}(G)}$ is based at $S_{0}$. If we fix $\xi_{\Gamma}^{\mathrm{Aut}(G)}$
and range over all $h:S_{0}\rightarrow S$ the resulting $\xi^{\mathrm{Aut}(G)}$ classes range over
the orbit $\{\xi^{\mathrm{Aut}(G)}\circ\mu^{-1}$, $\mu\in M_{\mathfrak{s}}\}$.

Next, we pick a standard generating set $\mathcal{W}=\left(  w_{1},\ldots
w_{l}\right)  $ of $\Gamma_{0}$ (see, for instance, \cite{BrCoIz1}). Each monodromy determines a \emph{generating
vector} of the $G$ action:
\begin{equation}\label{eq-W2V}
\mathcal{V}=\langle\xi,\mathcal{W}\rangle = \left(  \xi(w_{1}),\ldots\xi(w_{l})\right). 
\end{equation}
As in the case of monodromies, we have $\left\vert \mathrm{Aut}(G)\right\vert$ 
distinct vectors in the $\mathrm{Aut}(G)$ class $\mathcal{V}^{\mathrm{Aut}(G)}$.
The $\mathrm{Aut}(G)$ classes of generating vectors (classes of monodromies), 
serves as a \emph{computable vertical invariant}.

Observe that $M_\mathfrak{s}$ acts upon the generating vector classes $\mathcal{V}^{\mathrm{Aut}(G)}$ 
via the left action of $M_\mathfrak{s}$ on monodromy classes  $\xi^{\mathrm{Aut}(G)}$, namely 
\begin{equation}\label{eq-GVmodaction} 
\mu\cdot \mathcal{V}^{\mathrm{Aut}(G)} 
= \langle\mu\cdot\xi^{\mathrm{Aut}(G)},\mathcal{W}\rangle 
= \langle \xi^{\mathrm{Aut}(G)}\circ\mu^{-1},\mathcal{W}\rangle. 
\end{equation}
We are now ready to define the moduli space of $G$ actions with signature $\mathfrak{s}$.

\begin{definition}\label{def-SGs}
Let $\Gamma_{0}$, $\mathcal{W}$, $\mathcal{T}_{\mathfrak{s}}$, $M_\mathfrak{s}$, and $\mathcal{M}_\mathfrak{s}$ be as
above, and let $\mathcal{F}_\mathfrak{s}$ be the (finite) set of generating vector classes
$\mathcal{V}^{\mathrm{Aut}(G)}$ of $G$ with signature $\mathfrak{s}$. Let $M_{\mathfrak{s}}$ act upon
$\mathcal{T}_{\mathfrak{s}}$ and $\mathcal{F}_\mathfrak{s}$ 
as described in equations \eqref{eq-Teichmodaction} and \eqref{eq-GVmodaction}.
Then,
\begin{equation}\label{eq-Bs}%
\mathcal{S}_{G,\mathfrak{s}}=\left(  \mathcal{T}_{\mathfrak{s}}\times\mathcal{F_\mathfrak{s}}\right)  /M_{\mathfrak{s}}
\end{equation}
is the \emph{moduli space of conformal equivalence classes of $G$ action pairs with
signature $\mathfrak{s}$}. The natural projection
\begin{equation}\label{eq-pBM}
p:\mathcal{S}_{G,\mathfrak{s}}=\left(  \mathcal{T}_{\mathfrak{s}}\times\mathcal{F}_\mathfrak{s}\right)
/M_{\mathfrak{s}}\rightarrow\mathcal{M}_{\mathfrak{s}} 
\end{equation}
is the orbifold projection of conformal equivalence classes of $G$ action pair classes  
$\left(S,\epsilon\right)^{\mathrm{Aut}(G)}$ to conformal equivalence classes of quotient orbifolds $T=S/G$.

Each connected component of $\mathcal{S}_{G,\mathfrak{s}}$ is called a \emph{(closed) stratum} or a \emph{modular companion component}
of the moduli space of $G$ action pairs with signature $\mathfrak{s}$. As in Proposition  \ref{prop-orbicover}, we may write
\begin{equation}\label{eq-modcomp-comp}
  \mathcal{S}_{G,\mathfrak{s}}=\bigcup_i\mathcal{S}^i_{G,\mathfrak{s}}
\end{equation}
as a decomposition into modular companion components whose projection maps   
\begin{equation}\label{eq-modcomp-map}
  p_i=p:\mathcal{S}^i_{G,\mathfrak{s}}\rightarrow \mathcal{M}_{\mathfrak{s}}
\end{equation}
are complex orbifold covering spaces. 
\end{definition}

\begin{remark}\label{rk-reg-sing}
As an orbifold, the singular set $\mathcal{M}_{\mathfrak{s}}^{sing}$ of $\mathcal{M}_{\mathfrak{s}}$ is the image of
points in $\mathcal{T}_{\mathfrak{s}}$ that have extra automorphisms. It is a closed, analytic
subvariety of $\mathcal{M}_{\mathfrak{s}}$, the regular points are $\mathcal{M}%
_{\mathfrak{s}}^{\circ}=$ $\mathcal{M}_{\mathfrak{s}}- \mathcal{M}_{\mathfrak{s}}^{sing}$. Normally, there
are no automorphisms of the orbifold surfaces in $\mathcal{M}_{\mathfrak{s}}^{\circ}$ unless we have
non-maximal actions, such as Type $2$ actions discussed in \cite{BrCoIz2}. 
Over the regular set $\mathcal{M}_{\mathfrak{s}}^{\circ}$, the fibres of $p:\mathcal{S}_{G,s}%
\rightarrow\mathcal{M}_{\mathfrak{s}}$ have the same cardinality and the same monodromy
structure, namely the $M_{\mathfrak{s}}$ permutation structure of $\left\{  \mathcal{V}%
^{\mathrm{Aut}(G)}\right\}  $. There is a similar statement for each of the modular companion components. 
\end{remark}

\subsection{Modular companions}\label{subsec-modcomp}
We now turn to the vertical invariants of $(G,\mathfrak{s})$ action moduli spaces. 
\begin{definition}
In the mapping \eqref{eq-pBM} $p:\mathcal{S}_{G,\mathfrak{s}}\rightarrow\mathcal{M}_{\mathfrak{s}}$
the preimage $p^{-1}(\overline{\rho})$, $(\overline{\rho}=M_{\mathfrak{s}}\rho)$
corresponds to the set of surfaces $S$ with a $(G,\mathfrak{s})$ action and whose quotient orbifolds are conformally equivalent to $S/G=\mathbb{H}%
/\rho(\Gamma_{0})$. Any two such surfaces are called \emph{modular companions} if they belong to the same connected component of  $\mathcal{S}_{G,\mathfrak{s}}$.
\end{definition}

\begin{remark}\label{rk-modcompdef}
Suppose that $S_1$ and $S_2$ are modular companions. By definition, we have $(G,\mathfrak{s})$ actions
$\epsilon_1$ and $\epsilon_2$ on $S_1$ and $S_2$,  respectively, and a quotient orbifold $T$ 
satisfying $S_1/G=T=S_2/G$.  
Moreover, it can be shown that the monodromies $\xi_1, \xi_2$ defining the surfaces and the actions satisfy
\[
\xi_2 =\omega\circ\xi_1\circ\mu^{-1}, \mathrm{\text{ }for\text{ }some\text{ }} \omega \in \mathrm{Aut}(G), \mu \in \mathrm{Aut}^+(\Gamma).
\]
and that the actions are topologically equivalent. The actions may not be conformally equivalent 
even though they define the same quotient surface. There is a ``combinatorial'' obstruction to 
conformal equivalence, which we explore in  Sections \ref{sec-tiling}, \ref{sec-S4}, and \ref{sec-Cayley}, using Cayley graphs.
However, since they lie in the same connected component of  $\mathcal{S}_{G,\mathfrak{s}}$, 
one action can be deformed to the other through a one parameter family of $(G,\mathfrak{s})$ actions.
\end{remark}

\begin{remark}\label{rk-2defs}
Modular companions were introduced in \cite{BrCoIz2} and we began a more serious study of them in this paper. 
Initially, we defined the relationship for conformally inequivalent pairs of surfaces, but we now allow 
conformally equivalent surfaces to be modular companions so that modular companionship is an equivalence relation. 
By the discussion in  \cite{BrCoIz2} and the discussion in the preceding Remark, the two definitions
(algebraic and topological) are equivalent with this trivial adjustment. 
We must remind ourselves that the modular companionship relation is always 
in reference to some $(G,\mathfrak{s})$ pair, and that we are comparing 
two action pairs $(S_1,\epsilon_1)$,  $(S_2,\epsilon_2)$ or two monodromy pairs $(T,\xi_1)$, $(T,\xi_2)$. 
\end{remark}

\subsection{Relation to equisymmetric strata}\label{subsec-eqstrata}

If $(F)$ is a conjugacy class of finite subgroups of the modular group $M_\sigma$,  then the equisymmetric stratum defined by $(F)$ is
\[
\mathop{\overset{\text{ }\circ}{\mathcal{M}}}{\vphantom{m}_\sigma^{(F)}}= \{[S]: \Sigma(S)=(F)\},
\]
where $[S]$ is the point in $\mathcal{M}_\sigma$ determined by $S$.  The closure of a non-empty stratum equals 
\[ 
\mathcal{M}_\sigma^{(F)} = \{[S]: \Sigma(S)\ge (F)\},
\]
where $(F_1)\le(F_2)$ means $F_1$ is contained in a conjugate of $F_2$.  We ignore the case of empty strata, which correspond to non-maximal actions. 

Considering the points of the action stratum $\mathcal{S}^i_{G,\mathfrak{s}}$ to be represented by subgroup pairs $(S,\epsilon(G))$, we get a map  
\[ 
\Theta :\mathcal{S}^i_{G,\mathfrak{s}} \rightarrow \mathcal{M}_\sigma,\text{  } \Theta(S,\epsilon(G)) = [S],
\]
where $[S]$ is the conformal equivalence class of $S$. 
In \cite{BrCoIz1} it was demonstrated that if the corresponding equisymmetric stratum is non-empty, then the image of $\Theta$ is the closure of the stratum in $\mathcal{M}_\sigma$ (also \cite{GDH} for surjectivity). 
Also in \cite{GDH}, it was proven that the map $\Theta$ is a normalization of the stratum closure, as well as developing a criterion for computing the non-normal points of the stratum closure. This was utilized in \cite{HPCRH} to compute many examples of non-normal stratum closures.  

One use of the branched cover \eqref{eq-modcomp-map} is to determine the topology of the action strata, as done in the one dimensional case in \cite{BrCoIz1,BrCoIz2}. The description of the topology sheds some light on the topology of the branch locus. A next step would be to determine the topology of strata corresponding to actions branched over 5 points on a sphere. However, that is beyond the scope of this paper.       

 \begin{example}\label{ex-S3-2233-2}
We continue Example \ref{ex-S3-2233-1}. 
It can be shown that the space of quotients $\mathcal{M}_{\mathfrak{s}}$  is a sphere with two punctures and one cone point of order $2$, though the details are beyond the scope of this paper. 

Now, consider the following braid operations on $(\gamma _{1},\gamma _{2},\gamma _{3},\gamma _{4})$.  
\begin{eqnarray*}
\Phi_{1,2}: (\gamma _{1},\gamma _{2},\gamma _{3},\gamma _{4}) &\rightarrow& (\gamma _{2},\gamma _{1}^{\gamma _{2}},\gamma _{3},\gamma _{4})  \\
\nonumber 
\Phi_{2,3}: (\gamma _{1},\gamma _{2},\gamma _{3},\gamma _{4}) &\rightarrow&   (\gamma _{1},\gamma _{3},\gamma _{2}^{\gamma _{3}},\gamma _{4}) \\
\nonumber 
\Phi_{3,4}: (\gamma _{1},\gamma _{2},\gamma _{3},\gamma _{4}) &\rightarrow&   (\gamma _{1},\gamma _{2},\gamma _{4},\gamma _{3}^{\gamma _{4}})
\end{eqnarray*}
It is well known that the modular group $M_\mathfrak{s}$ is the set of outer automorphisms of 
$\Gamma_\mathfrak{s}$  induced by combinations of the above braid operations and their inverses that preserve the signature. 
In particular $\Phi_{1,2},\Phi^2_{2,3},\Phi_{3,4} \in M_\mathfrak{s}$. 
A braid operation $\Phi^{-1}$ of $M_\mathfrak{s}$ 
acts upon a generating vector $\mathcal{V}= (c_1,c_2,c_3,c_4)$, by applying the formula for $\Phi$ to  $\mathcal{V}$ (change $\gamma$ to $c$).
See \cite{BrCoIz1} for details.

Now let $\mathcal{V}_1 =(x,x,y,y^{-1})$ and $\mathcal{V}_2 =(x,xy,y,y)$ be the generating vectors
from Example \ref{ex-S3-2233-1}. By easy hand calculations we see that
\[ 
\Phi^{-1}_{1,2}\cdot \mathcal{V}_1=\mathcal{V}_1 \text{ and } \Phi^{-1}_{1,2}\cdot \mathcal{V}_2=(xy,xy^2,y,y)=\mathrm{Ad}_{y}\cdot\mathcal{V}_2.
\]   

So $\Phi_{1,2}$ and $\Phi^{-1}_{1,2}$  fix both $\mathcal{V}_1^{\mathrm{Aut}(G)}$ and $\mathcal{V}_2^{\mathrm{Aut}(G)}$.
Similar calculations also show that $\Phi_{3,4}$ acts trivially. Finally,  $\Phi^2_{2,3}$ and $\Phi^{-2}_{2,3}$ interchange 
$\mathcal{V}_1^{\mathrm{Aut}(G)}$ and $\mathcal{V}_2^{\mathrm{Aut}(G)}$, so that they define modular companions. 
It follows that the action moduli space is a degree 2 orbifold cover of $\mathcal{M}_{\mathfrak{s}}$.

The case of any planar action with signature $(0;m_1, m_2,m_3, m_4)$ can be handled similarly, taking symmetries of the signature into account. 
Again, see \cite{BrCoIz1} for additional details.
\end{example}

\section{Equivariant tilings of surfaces}\label{sec-tiling}

We may view the pair $(T,\xi)$ as a recipe for constructing an equivariantly tiled surface
$S$ from ``puzzle pieces'', as described in
\cite{BrPaWo}, Section 6.4. The process, which we describe next, and the
subsequent construction of a dual Cayley graph embedded on the surface (Section \ref{sec-Cayley}), will
give us geometrical tools for comparing modular companions. 
In this section, we discuss the general situation for arbitrary quotients $T=S/G$. 
In the next section (Section \ref{sec-S4}), we give detailed examples and calculations for spherical quotients with four branch points, 
which form equisymmetric strata of Teichm\"{u}ller dimension one.
The methods of Section \ref{sec-S4} can be easily extended to arbitrary signatures.

\subsection{Cut systems  of $T$}
Next, we examine cut systems on the quotient surface $T$. This analysis is inspired by the method of cuts 
used to create a Riemann surface upon which an algebraic function is properly defined. A reference for this method is \cite{JS}.  

 As in Section \ref{subsubsec-monolift}, denote by $T^\circ$ and $S^\circ$ the punctured surfaces obtained by removing the cone
points of $T$ and their $\pi_G$ pre-images in $S$. 
The map $\pi_G:S^\circ\rightarrow T^\circ$ is a conformal covering space, 
whose group of deck transformations is $\left\{\epsilon(g): g\in G\right\}$.
We use this cover in analyzing lifts of \emph{cut systems} on $T$, which we now introduce.   
Make a system of cuts in $T$ by means of an embedded graph $\mathcal{E}$,
that dissects $T$ into a open polygon $\mathcal{P}^\circ=T-\mathcal{E}$ with a piecewise smooth,
curvilinear boundary. 
For quick examples see Figure \ref{tikz-E1-4} in the next section. 
Our full requirements are in the following definition.

\begin{definition}
Let  $G$ act conformally on  $S$ and let $T=S/G$ be the quotient orbifold. Let $\mathcal{E}$ be an embedded graph
in $T$  with $k$ edges such that the vertex set contains all the cone points of $T$; the $k$ arcs are smooth, 
with definite tangents at their end points; there are no arc intersections except at vertices; 
and the arc intersections at vertices have distinct tangents. A regular point of $T$ is not allowed to be a terminal vertex. 
We also assume that the complement of $\mathcal{E}$ is an open disc with a piecewise smooth curvilinear boundary (the main point).  
Such an embedded graph will be called a \emph{cut system}.
\end{definition}

For brevity's sake, we limit our examples of cut systems to the four cases in Figure \ref{tikz-E1-4} for orbifold 
spheres with $4$ cone points, yielding one dimensional strata as advertised in this paper's title. 
The cone points are black nodes, and in some cases a white node, denoting a regular point, has been added.  
The only one dimensional strata we are missing is a torus quotient with one cone point, which we leave out because of space limitations. 
Our definition does allow for a positive quotient genus and arbitrary numbers of cone points.  
All these more general systems can be analyzed using the methods we develop here.
Models of the polygons corresponding to the cut systems in Figure \ref{tikz-E1-4} are given are in Figure \ref{tikz-P1-4}. 
The cone points in polygons are the black nodes. The polygons given are models only, and they look very different in their natural geometry.

\begin{remark}\label{rk-homology-tree}
Using the long exact sequence in homology of the pair $(T,\mathcal{E)}$ and the excision theorem, we see that $H_1(T)=H_1(\mathcal{E})$. 
So, in the genus zero case, $\mathcal{E}$ must be a tree. In higher quotient genus cases some edges will be loops. Observe that the number of trees on a finite set of points grows very fast, according to Cayley's formula. Not all possibilities for four cone points on a sphere are given in Figure \ref{tikz-E1-4}. 
\end{remark}

\subsection{Lifted polygons in the tiling of $S$}\label{subsec-tiling}   

We consider the lifts to $S$ of various objects in $T$ via the branched
covering $\pi_{G}:S\rightarrow T$. Let $\widetilde{\mathcal{E}}=\pi_G^{-1}(\mathcal{E})$ be the lift
of $\mathcal{E}$ to $S$. This lift induces a tiling on $S$ with the following characteristics:

\begin{enumerate}
\item The complement $S-\widetilde{\mathcal{E}}$ is a disjoint union of
open discs (open polygons), each of which is conformally equivalent to $\mathcal{P}_0 = T-\mathcal{E}$ via
$\pi_{G}$. Therefore, it defines a \emph{map} on $S$ in the standard graph theoretic sense.
The open polygons are permuted simply transitively by $G$. Select a distinguished open polygon 
$\widetilde{\mathcal{P}^\circ}\subset S-\widetilde{\mathcal{E}}$. Using the simple transitivity, 
we may label the open polygons via the disjoint union
\[
S-\widetilde{\mathcal{E}} =\bigcup_{g \in G}g\widetilde{\mathcal{P}^\circ}.
\] 
\begin{notation}
For convenience, if there will be no confusion, we drop the $\widetilde{\ }$ over the polygons.
\end{notation}

\item Let $e$ be an open edge of $\mathcal{E}$ and  $\tilde{e}$ any open edge lying over $e$, we say that $\tilde{e}$
is an \emph{edge of type $e$}. Likewise, if  $v\in \mathcal{E}$ is a cone point or free (unramified) vertex, then the 
vertices $\tilde{v}$ of  lying over $v$ are said to be of\emph{ type $v$}. Vertices in $\widetilde{\mathcal{E}}$ of the same type inherit the  
cone point labeling in $\mathcal{E}$. If there is more than one regular vertex, we need to label them as well.

\item Every edge $\tilde{e}$ in $\widetilde{\mathcal{E}}$ connects exactly two distinct
vertices in $\widetilde{\mathcal{E}}$, at least in the sphere quotient case.
In the sphere quotient case, according to Remark \ref{rk-homology-tree}, $\mathcal{E}$ is a tree and so has no loops. Therefore, the endpoints of $e$ 
are distinct, so that the endpoints of $\tilde{e}$ lie in different fibres of  $\pi_G$.

If $T$ has genus $\ge 1$ then $\mathcal{E}$ may have loops and $\widetilde{\mathcal{E}}$ may also have loops. 
This can be repaired by bisecting the loops in  $\mathcal{E}$  with regular points at the cost of additional complexity. 
Because of length limitations of this paper we skip this interesting topic. The first non-trivial case is when $T$ is a torus with a cone point, as considered in \cite{BrCoIz2}.

\item In fact, $\widetilde{\mathcal{E}}\rightarrow\mathcal{E}$ is a regular branched covering of
graphs with deck transformation group $\left\{\epsilon(g) :  g\in G\right\} $, 
restricted to the edges. The action on the set of closed edges is free.  
For, suppose  $g\overline{e}=\overline{e}$ as a set, for some closed edge $\overline{e}$
in $\widetilde{\mathcal{E}}$. Then $g$ must either fix the endpoints or
switch them. If an end point $v$ is fixed by $g$, then, as  $g$ acts as a
rotation at $v$, it cannot fix the tangent of $\overline{e}$ at $v$. If $g$ switches the
end points then it must fix a point in the interior of $e$. However, by
construction, any such point must be a vertex. 
We assumed that $\overline{e}$ had distinct endpoints, though that need not happen. 
However, a similar argument still shows that $g$ is trivial.

\item By the construction of  $\mathcal{E}$, and the structure of the orbifold cover $\pi_G$, every 
lift $\tilde{v}$ of a vertex $v\in \mathcal{E}$ has a neighbourhood $U$ of the following form.
There is a small disc-like neigbourhood $V$ of $v$ in $T$, with $U=\pi_G^{-1}(V)$, such that in local complex coordinates,  
$\pi_G$ has the form  $z\rightarrow z^m$ ($m=1$ for a regular point). 
Also, the stabilizer in $G$ of $\tilde{v}$, of order $m$,
acts by $z\rightarrow \zeta^jz$ for some primitive $m$th root of unity $\zeta$. Next, 
by making $V$ small enough, the $k_v$ edges of $\mathcal{E}$, incident with $v$, cut up $U$ into $k_v$ open sectors based at $v$.
Lifting to $S$, the neighbourhood $U$ will be cut up into $m\times k_v$ open sectors, based at $\tilde{v}$. 
Each sector has a label determined by 
\begin{equation}\label{eq-seclabel}
sectors_g = U\cap g\mathcal{P}^\circ  
\end{equation}
for an appropriate set of elements $g$ in $G$.
The connected components of $sectors_g$ are the sectors with label $g$, there may be more than one sector with the same label. 
The action of the $G$ stabilizer at $\tilde{v}$ induces a permutation of the sectors and hence a pattern on the labelling.     

\item In ``good'' cases, the closure in $S$ of each open polygon of $S-\widetilde{\mathcal{E}}$ is
an embedded, closed polygon in $S$. I.e., the closed polygon is a topological closed disc, there are no 
self-intersections on the boundary, and so the boundary is a topological circle with $2k$ edges and $2k$ vertices.
There are degenerate cases where this may fail, as described in the next Section (\ref{subsec-ribbon-bad-poly-lift}). 
In every case, we denote the closure of $\widetilde{\mathcal{P}^\circ}$ by $\widetilde{\mathcal{P}}$  
or just $\mathcal{P}$, if there will be no confusion. 

\item Let  $\widetilde{\mathcal{P}}^\star$, called a \emph{nipped polygon},  denote the closed polygon $\widetilde{\mathcal{P}}$ 
with the vertices removed. We will use  $\mathcal{P}^\star$ if there will be no confusion.
In the absence of an edge collapse, described in the next section (\ref{subsec-ribbon-bad-poly-lift}), 
each  nipped polygon  $g\mathcal{P}^\star$, is homeomorphically embedded in $S$, as a planar nipped polygon. 
Specifically, there is a plane convex polygon $\mathcal{P}_{plane}$ and a continuous map $\mathcal{P}_{plane} \rightarrow g\mathcal{P}$ which is a homeomorphism away from the vertices.

\end{enumerate}

\subsection{Boundaries, ribbon graphs, and degenerate lifts}\label{subsec-ribbon-bad-poly-lift} 
In items 6 and 7 above, we noted that lifting polygons could go wrong when closures of lifted polygons 
have self-intersections on the boundary, i.e., the closure is not homeomorphic to a closed disc. 
These degeneracies can be detected by \emph{crossover transformations} (side-pairings) and \emph{sector labelling sequences}, 
both of which we introduce and discuss in the next section (\ref{subsec-crossover-sectorlabelling}).
If an open lifted polygon has a closure that is not an embedded polygon, we shall call the closure a \emph{degenerate lift}.
Because of the $G$ action, all closures of lifted polygons have homeomorphic closures, and so they are all non-degenerate or all degenerate.

\begin{remark}\label{rk-avoid-degen}
Other than simplicity of the tiling, a specific reason for avoiding degenerate lifts is avoiding 
degeneracies in the associated  Cayley graphs, introduced in Section \ref{sec-Cayley}.
Moreover, the tilings without degeneracies may be directly used for computing the homology representations of $G$ on $H_1(S)$.  
Once a monodromy is chosen, the degeneracy properties of lifts depend only on the choice of $\mathcal{E}$. 
\end{remark}

\subsubsection{Ribbon graphs}\label{subsubsec-ribbon}
In order to analyse the degeneracies of lifted polygons and to analyze how the various lifted polygons fit together, 
we need to be a bit more precise in describing the boundary and its orientation.
We can do this by constructing and lifting ribbon graph tubular neighbourhoods of $\mathcal{E}$. 
The ribbon graph neighbourhood $\mathcal{R}_\epsilon$ of $\mathcal{E}$ is an open tubular neighbourhood of $\mathcal{E}$, constructed as follows. 
Each node of the cut-system will be covered by a small disc, and each arc will be covered by a narrow ribbon. 
In  Figure \ref{tikz-ribbon} we show a ribbon graph neighbourhood for $\mathcal{E}_1$, which in turn, is shown in Figure \ref{tikz-E1-4}.  
The discs are red and the ribbons are green to help distinguish them. 
The parameter $\epsilon$ is a measure of the width of the tubular neighbourhood, and as $\epsilon \rightarrow 0$, $R_\epsilon$ 
``converges'' to $\mathcal{E}$.

Lift the complement $\mathcal{R}_\epsilon-\mathcal{E}$ to $S$. Since $\mathcal{R}_\epsilon-\mathcal{E}$ is contained entirely within 
$T-\mathcal{E}$, the lift is a conformal homeomorphism of $\mathcal{R}_\epsilon-\mathcal{E}$ to its image in each open polygon $g\mathcal{P}^\circ$ in $S$:
\[
\mathcal{R}_\epsilon-\mathcal{E}\leftrightarrow \pi_G^{-1}(\mathcal{R}_\epsilon-\mathcal{E})\cap g\mathcal{P}^\circ.
\]
The lifted complement  $\pi_G^{-1}(\mathcal{R}_\epsilon-\mathcal{E})\cap g\mathcal{P}^\circ$ has $2k$ portions that are green ribbons and $2k$ red sectors. 
Each green portion sweeps out an open ribbon that runs parallel to one of two sides of a lifted edge. Each red portion is a open sector of a disc centered 
at an endpoint of a lifted edge. Because we imposed distinct tangents at the endpoints of edges 
and since only cone points are allowed to be terminal edges, then each red sector is a proper sector
(sector angle $\le 2\pi$) of the disc surrounding a lifted vertex. For each vertex $v \in \mathcal{E}$ there are $k_v$  red sectors.  
The red and green portions may be pieced together to get a ribbon $S^1\times (0,1)$ such that the circle factor 
$S^1$ has $2k$ distinguished points (the vertices) corresponding to the tips of the red sectors. The portions of $S^1$ 
between vertices correspond to $2k$ edges of $\widetilde{\mathcal{E}}$ in the closure of a lifted polygon.

\begin{remark}\label{rk-ribbonboundary}
Note that $\mathcal{P}_\epsilon = T-\mathcal{R}_\epsilon$  is an approximation (strong deformation retract) to the open polygon $T-\mathcal{E}$. 
The closed ribbon graph neighbourhood $\overline{R}_\epsilon$ and $\mathcal{P}_\epsilon$ satisfy: 
\begin{eqnarray*}
  \overline{\mathcal{R}_\epsilon} \cup \mathcal{P}_\epsilon &=& T,  \\
   \overline{\mathcal{R}_\epsilon} \cap \mathcal{P}_\epsilon &=& \partial \overline{\mathcal{R}_\epsilon}  = \partial \mathcal{P}_\epsilon.  
\end{eqnarray*}
The common boundary may be identified with the $S^1$ factor of $\mathcal{R}_\epsilon-\mathcal{E}$ discussed above.
Thus, we can describe the counterclockwise list of edges and vertices of the closed model of $\mathcal{P}$ by traversing 
the boundary of $\partial \mathcal{P}_\epsilon$ in a counterclockwise fashion  
or the boundary $\partial \overline{\mathcal{R}_\epsilon}$ in a clockwise fashion.
    
As described in Remark \ref{rk-homology-tree}, we have $H_1(\mathcal{R}_\epsilon) \simeq H_1(\mathcal{E}) \simeq H_1(T)$. 
In the genus zero case $\mathcal{R}_\epsilon$ will be a disk. 
Boundary behaviour can be easily determined from a visual examination of the ribbon graph neighbourhood. 
\end{remark}

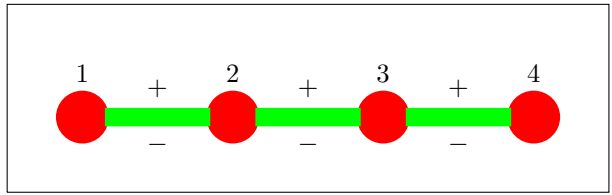
\begin{figure}[h!]
\begin{center}

\begin{tikzpicture}[scale=0.5, inner sep=1mm]

\node (A) at (-6,0) [shape=circle, fill=red, minimum size=20] [label=$1$] {};
\node (B) at (-2,0) [shape=circle, fill=red, minimum size=20] [label=$2$]{};
\node (C) at (2,0) [shape=circle,  fill=red, minimum size=20] [label=$3$]{};
\node (D) at (6,0) [shape=circle, fill=red, minimum size=20] [label=$4$]{};
\node (E1+) at (-4,0.1) [label=$+$] {};
\node (E2+) at (0,0.1) [label=$+$] {};
\node (E3+) at (4,0.1) [label=$+$] {};
\node (E1-) at (-4,-1.4) [label=$-$] {};
\node (E2-) at (0,-1.4) [label=$-$] {};
\node (E3-) at (4,-1.4) [label=$-$] {};

\draw [green,line width=7] (-5.4,0) to (-2.6,0);
\draw [green,line width=7] (-1.4,0) to (1.4,0);
\draw [green,line width=7] (2.6,0) to (5.4,0);

\draw [thin] (-8,-2) to (8,-2) to (8,3) to (-8,3) to (-8,-2) ;

\end{tikzpicture}

\end{center}
\caption{Ribbon graph for $\mathcal{E}_1$, with orientation markings}
\label{tikz-ribbon}
\end{figure}

\subsubsection{Degenerate lifts}

For each edge $e$ in $\mathcal{E}$, the edge $e$  bisects the corresponding green ribbon into two halves.  
Two different situations can occur. First case:  the closures of the lifts of the two halves do not meet 
except possibly joining at a red sector lift. In this case the lifted, nipped polygon $\widetilde{\mathcal{P}}^\star$ 
has two distinct open edges lying over $e$. Second case: the two lifted halves meet along a common lifted arc,\
and the closure of the lifted polygon is  glued to itself along the lifted arc. The second situation, 
which we call an \emph{edge collapse}, is a degeneracy which can be detected by a crossover calculation, see Proposition \ref{prop-edgecollapse}.

Another type of (mild) edge degeneracy is that two lifted polygons can meet in more than one edge. This frequently happens when we have cone points of order 2. In the dual Cayley graph (Section \ref{sec-Cayley}), this means that two nodes may be connected by multiple edges, which is not allowed in a standard group Cayley graph (Section \ref{subsubsec-groupCayley1}). Multiple intersections can also be detected by crossover calculations.  See Proposition \ref{prop-multiedge} and Corollary \ref{cor-XTorder2}.

Even if there is no edge collapse, we can have a \emph{vertex collapse} when the boundary has vertices identified, 
namely, several lifts of a single red sector have a common sector vertex.  
As with edge collapses, a vertex collapse can be detected using a series of crossover calculations. 
The details are given in Section \ref{subsubsec-sectorlabelling}.

\subsection{Crossover transformations and sector labelling}\label{subsec-crossover-sectorlabelling}
As previously discussed, we may label the open polygons in $S$ by the elements of $G$, namely $g \leftrightarrow g\mathcal{P}^\circ$.
If a base point $v_0 \in T-\mathcal{E}$ has been chosen, then the polygons may also be identified by lifts of the basepoint, 
namely,
\[
g \tilde{v}_0 = \pi^{-1}_G(v_0)\cap g\mathcal{P}^\circ.
\] 
We want to know how the labelling changes as we cross an edge or make a circuit around a vertex in $S$.  
The labeling changes can be computed from the structure of $\mathcal{E}$, the  monodromy $\xi=\xi_{\tilde{v}_0}$,  
and the corresponding generating vector $\mathcal{V}$. We will employ \emph{crossover transformations} and 
\emph{sector labelling sequences}, discussed next. 

\subsubsection{Crossover transformations}\label{subsubsec-crossover1}
For each arc or loop $e$ in $\mathcal{E}$ we construct a \emph{crossover loop}, i.e., a simple loop $\delta_e$ in $T^\circ$ 
based at the interior point $v_0$; that crosses $e$ transversally, exactly once, 
at a point $v_e$ in the interior of $e$; and that does not intersect any other arc or vertex of 
$\mathcal{E}$. The loop shall be used to determine label changes and construct side-pairing transformations. 
The crossover loop is unique  up to homotopy and orientation.
Examples of such loops are given by the red curves and their inverses in 
Figures \ref{tikz-XTloops1} and \ref{tikz-XTloops2}, where  $v_0=\infty$. 

A crossover loop may be constructed as follows. Find a small, disk-like neighbourhood  $D_e$ of $v_e$
which does not meet any part of $\mathcal{E}$, except the interior of $e$, and is bisected by $e$. 
Construct a small bisector $b_e$ of $D_e$ that crosses $e$ transversally, exactly once, at $v_e$. 
Let $U_e$ be the tubular neighbourhood of $e$ defined by the ribbon graph.
The open set $T-\mathcal{E}$ meets $U_e$ in two halves $U_e^+$ and $U_e^-$. 
(How to determine $+/-$ is discussed later in Section \ref{subsubsec-e-orientation}).  
We will assume that $D_e \subset U_e$ and that  $b_e$ meets the boundary of $D_e$ (but not $e$) in the points
$v^+ \in U_e^+$ and  $v^- \in U_e^-$.  Now construct a concatenation of paths: from $v_0$ to $v^+$ 
within the open polygon $T-\mathcal{E}$, then along the bisector from 
$v^+$  to $v^-$, and then back to  $v_0$ within the open polygon $T-\mathcal{E}$. 
We may assume that $\delta_e$ has no self intersections except at the base point. 
If we want to be specific about the orientation of $\delta_e$ we set 
\[
 \delta^+_e=\delta_e, \text{   }  \delta^-_e=(\delta_e)^{-1}.
\]

\begin{definition}
Let all notation be as above and let  $\mathcal{C} = \{\delta_e : e \in \mathcal{E}\}$ be the set of geometric loops defined by crossover paths.
Then $\mathcal{C}$, considered as an undirected graph on $T$,  is called the \emph{dual graph} to $\mathcal{E}$  or the \emph{quotient Cayley graph} with respect to $\mathcal{E}$.  
\end{definition}

From our lifting discussion, we see that the lift $\tilde{\delta_e}$ satisfies
\begin{equation}\label{eq-XTdef1}
\tilde{\delta_e}(1) = \xi(\delta_e)\cdot\tilde{v}_0.
\end{equation}
Moving along the lift $\tilde{\delta}_e$, we first travel within $\mathcal{P}^\circ$, 
cross an edge $\tilde{e} \in \partial\mathcal{P}^\circ $ lying over $e$,
and then travel within the open polygon $\xi(\delta_e)\mathcal{P}^\circ$ to end up at the lift   
$\xi(\delta_e)\tilde{v}_0$ of the base point $v_0$. We define the \emph{crossover transformation} 
\begin{equation}\label{eq-XTdef2}
\tau_e =\xi(\delta_e) 
\end{equation}
as we ``cross over'' the lifted edge  $\tilde{e}$ from $\mathcal{P}^\circ$ to $\tau_e \mathcal{P}^\circ$.  
If we use the opposite orientation on $e$ then we use  $\delta_e^{-1}$ 
to cross over the ``other lift'' $\tilde{e}\in \partial\mathcal{P}^\circ$, lying over $e$,
yielding the crossover transformation $\tau_e^{-1}$. 
Note that if  $\tau_e$ is trivial, then the ``other'' lift will be the same and we have an edge collapse.

If we start within another polygon $g\mathcal{P}^\circ$ we simply 
transport all the lifted items by $g$. Our lifts now start at $g\tilde{v}_0$ and our crossover transformation
is $g\tau_eg^{-1}$, according to our previous discussion on monodromies, see equation \eqref{eq-conjlift}. 
We cross over from $g\mathcal{P}^\circ$ to $(g\tau_e g^{-1}g)\mathcal{P}^\circ=(g\tau_e)\mathcal{P}^\circ$. 
Since $\mathcal{E}$ has $k$ edges with two orientations each, each lifted open polygon has at most $2k$ crossover 
transformations and is bounded by edges from at most $2k$ other polygons. 

\begin{definition}\label{def-crosoverseq}
  Assume we have some ordering $(e_1,\ldots,e_{2k})$ of the $2k$ \emph{oriented edges} $e_j$ of $\mathcal{E}$. 
Then, $XTS=(\tau_{e_1},\ldots, \tau_{e_{2k}})$ is called the \emph{crossover transformation sequence}, with respect to this ordering.
\end{definition}

\begin{remark}\label{rk-standard edge order}
Unless there is a good reason otherwise, the standard ordering of the oriented edges of $\mathcal{E}$ 
is that induced by a counterclockwise circuit about the open polygon $T-\mathcal{E}$, starting at a specified vertex. 
See Remark \ref{rk-ribbonboundary} about computing the order from the boundary of the ribbon graph.
\end{remark}

Using the crossover sequence, we can now characterize edge collapses and when two polygons closures meet along multiple edges.

\begin{proposition}\label{prop-edgecollapse}
Let all notation be as above, and let $e$ be an edge of $\mathcal{E}$. Then there is an edge collapse at every 
edge lying over $e$ if and only if $\tau_e$ is trivial.
\end{proposition} 

\begin{proposition}\label{prop-multiedge}
Let all notation be as above, let $g\mathcal{P}^\circ$ and  $h\mathcal{P}^\circ$ be any two distinct lifted open polygons. 
Then the closures of $g\mathcal{P}^\circ$ and  $h\mathcal{P}^\circ$  meet in at least one edge if and only if $g^{-1}h$
is in the crossover sequence $(\tau_{e_1},\ldots, \tau_{e_{2k}})$. 
The number of  edges in the intersection equals the number of repetitions of $g^{-1}h$ in the crossover sequence.
\end{proposition} 

\begin{corollary}\label{cor-XTorder2}
If the crossover sequence contains an involution, then every polygon meets at least one other polygon in two or more edges. 
\end{corollary}

\subsubsection{Side-pairing transformations}\label{subsubsec-sidepairing}
By construction, the closures of $g\mathcal{P}^\circ$ and   $g\tau_e\mathcal{P}^\circ$ contain
a lift $\tilde{e}$ of $e$ in their intersection. Therefore, $\sigma_{\tilde{e}}=g\tau_e^{-1}g^{-1}$ maps 
the lift $\tilde{e}$ to a possibly different lift ${\tilde{e}}^\prime$ contained in $\partial(g\mathcal{P}^\circ)$, and lying over $e$. 
Such a map is called a side-pairing transformation. The sides are different if and only if $\sigma_{\tilde{e}}$ is not trivial.
In the nontrivial case, exactly two sides are paired, namely the two edges lying over $e$. 
The inverse $\sigma^{-1}_{\tilde{e}}$ pairs  ${\tilde{e}}^\prime$ with $\tilde{e}$.
Because $G$ acts freely on the lifted edges, $\sigma_{\tilde{e}}$ is uniquely determined. 
The side-pairing transformations may be used to construct the quotient $T$, including conformal structure and cone point structure,
from the closure of the open polygon $g\mathcal{P}^\circ$.  
In the distinguished polygon,   $\mathcal{P}^\circ$,   $\sigma_{\tilde{e}}= \tau^{-1}_e$.
\vskip 1pt

{Note: This ``construction'' is really just the Poincar\'{e} polygon theorem, adjusted for non-geodesic sides. 
Some care is required at the cone points of $T$. A good reference for the Poincar\'{e} polygon theorem is \cite{Mask}.

The details of the side pairing transformation,  $\sigma_{\tilde{e}}$, merit special attention when it is an involution. 
The situation is summarized in Corollary \ref{cor-XTorder2}. 
The boundary $\partial\mathcal{P}$, containing  $\tilde{e}$, also contains the paired edge
${\tilde{e}\,}^\prime=\tau_e^{-1}\cdot \tilde{e}=\tau_e\cdot\tilde{e}$, and it is different from $\tilde{e}$. 
Thus, both  $\mathcal{P}$ and  $\tau_e\mathcal{P}$ contain $\tilde{e}$ and ${\tilde{e}\,}^\prime$ in their boundaries. 
The two polygons can be glued together along their common edges $\tilde{e}$ and  ${\tilde{e}\,}^\prime$.
If $\tilde{e}$ and  ${\tilde{e}\,}^\prime$ are disjoint then we can build a cylinder within $\mathcal{P}\cup \tau_e\mathcal{P}$, 
by transporting $\tilde{e}$ along the lift of $\delta_e^2$, a closed loop contained within $\mathcal{P}\cup \tau_e\mathcal{P}$. 

Next, suppose that $\tilde{e}$ and  ${\tilde{e}\,}^\prime$ meet at an end point $v \in S$. 
If the endpoint is a white node, then, since $\pi_G$ is a local homeomorphism at $v$, there cannot
be two separate lifts of $e$, a contradiction. So now suppose that the common end point $v$ lies over a cone point. 
Lift a disk-like neighbourhood of $\overline{v}=\pi_G(v)$ to neighbourhood  $V$ of $v$. 
The set $V \cap (\mathcal{P}\cup \tau_e\mathcal{P})$ contains a cone with two open sectors, namely 
$V \cap \mathcal{P}^\circ$ and  $V \cap \tau_e\mathcal{P}^\circ$. 
When we add in $V \cap \tilde{e}$ and ${V \cap \tilde{e}\,}^\prime$ we get a cone, 
which must comprise the entirety of $V$. Alternatively, we know that 
$\partial V$ is a circle, but that $\partial V \cap (\mathcal{P}\cup \tau_e\mathcal{P})$ 
is a closed subset which is also a circle, a contradiction, unless $V = V \cap (\mathcal{P}\cup \tau_e\mathcal{P})$. 
Therefore, a neighbourhood of the cone point looks like the open unit disc, 
with the open upper (lower) half disc corresponding to points of 
$V \cap \mathcal{P}^\circ$ ($V \cap \tau_{\tilde{e}}\mathcal{P}^\circ$, respectively),
and the intervals $(-1,0]$ and $[0,1)$ correspond to $V \cap \tilde{e}$ and ${V \cap \tilde{e}\,}^\prime$. 
The transformation $\tau_e$ corresponds to a half turn about the origin, and generates the $G$-stabilizer of $v$. 

\subsubsection{Determining orientation}\label{subsubsec-e-orientation} 
When there is no edge collapse, the nipped, lifted polygon is locally a closed half space 
at each point of the interior of a lifted edge $\tilde{e}$. This induces an orientation on the lifted edges 
in the standard way of differential geometry. The orientation may be used to place an orientation on the crossover loops. 
However, we will use the ribbon graph to simplify the determination of the orientation of the crossover loops in equations 
\eqref{eq-XTdef1} and \eqref{eq-XTdef2}.

In the ribbon graph neighbourhood, the tubular neighbourhoods of the edges have two boundary components. 
Starting at an agreed upon vertex, move along the boundary of the ribbon graph in the clockwise direction. For each edge, $e$, 
we will travel through the two boundary components defined by the edge, separated from each other in the sequence, 
except at terminal endpoints of $\mathcal{E}$. The first boundary component we encounter we  mark with a $+$ 
and the second with a $-$.  After all the marking is complete, each edge will have a $+$ sign on one side and 
a $-$ sign on the other side. We denote the corresponding edges by $\tilde{e}^+$ and $\tilde{e}^-$. 
See Figure \ref{tikz-ribbon} for an illustration. 

Now we need to determine the orientation of the crossover loops $\delta_e$. The loop $\delta^+_e$, crosses over  
$\tilde{e}^+$. In the construction of the crossover transformation, the loop $\tilde{e}^+$ travels from within the polygon, 
across $\tilde{e}^+$, and into the adjoining polygon. 
Therefore on the ribbon graph we chose an orientation of $\delta_e$ so that it crosses from the $+$ to the $-$ side of the edge $e$. For  $\tilde{e}^-$ we must select the orientation of $\delta_e$ so that it crosses from  the $-$ side to the $+$ side of the edge $e$.

\subsubsection{Sector labelling sequences}\label{subsubsec-sectorlabelling}
Let us investigate if there are any vertex collapses, which, fortunately, can be determined by crossover transformations.
Let $\tilde{v}$ be a lift of a vertex $v \in \mathcal{E}$,  of valence $k_v$.  
Let $D$ be a small $G_{\tilde{v}}$-equivariant disc surrounding $\tilde{v}$ in $S$. Starting at a selected sector labelled with $h \in G$,  
let us move around $\tilde{v}$, along $\partial D$, in a counterclockwise direction. 
This direction is consistent with the direction of rotation of $c_v$, the given generator of the stabilizer $G_{\tilde{v}}$.
If $v$ is a white node then $c_v=1$ and $m_v=o(c_v) = 1$.
As we move around the vertex $\tilde{v}$ we will encounter, in order, the lifts 
$\tilde{e}_1,\tilde{e}_2,\ldots,\tilde{e}_{k_v},\ldots$ with $m_v$ repetitions. 
The intersection of these edges with the disc $D$ will be called \emph{spokes}.
The sequence $e_1,\ldots,e_{k_v}$ is a temporary labelling of the edges, of $\mathcal{E}$ issuing from $v$, 
in counterclockwise order and taking edge orientation into account.
As we move from the first sector to the second sector, our label changes from  $h$ to $h\tau_{e_1}$. 
Continuing this process we get a sequence, of $k_v\times m_v$ labels
\begin{equation}\label{eq-sectorseq}
h,h\tau_{e_1}, h\tau_{e_1}\tau_{e_2},h\tau_{e_1}\tau_{e_2}\tau_{e_3}, \ldots .
\end{equation}
These labels are the labels of the sectors, starting in the chosen sector and moving in a small counterclockwise circuit about $\tilde{v}$.  
Each new term in the sequence is obtained by right multiplying the preceding term by some crossover transformation $\tau_e$. Let us formalize this in a  
definition.

\begin{definition}
Let $\tilde{v} \in \widetilde{\mathcal{E}}$ be a vertex lying over a vertex $v$ of valency $k_v$ in  $\mathcal{E}$, 
let $m_v$ be the order of the stabilizer $G_{\tilde{v}}$, and let other notation be as above. Let $h \in G$, be the label for a chosen sector at $\tilde{v}$. 
Then the sequence of  $k_v\times m_v$ group elements given in \eqref{eq-sectorseq} is called the \emph{sector labelling sequence}, 
determined by the chosen sector. 
\end{definition}

By the construction  of the sector labelling sequence \eqref{eq-sectorseq}, we have:
\begin{proposition}
Let all notation be as above. Then, there is a vertex collapse at the vertex $\tilde{v}$ lying over $v$ in $\mathcal{E}$ if and only if there are repetitions in the sector labelling sequence \eqref{eq-sectorseq}.  
\end{proposition}

\section{Four cone points on a sphere}\label{sec-S4}
Given a monodromy $\xi$, and a tiling polygon $\mathcal{P}$, how do we determine
the quotient Cayley graph $\mathcal{C}$ and then the specific crossover transforms $\tau_e$.
Rather than giving a general discussion we will consider a small number of examples
with signatures  $(0;m_{1},m_{2},m_{3},m_{4})$, sufficient for our purposes.
We proceed with some examples of cut systems $\mathcal{E}$, then the polygons $\mathcal{P}$, 
and then the crossover transformations $\tau_e$.

\paragraph{Cut systems.}\label{subsubsec-cutsys}
For the signature $(0;m_{1},m_{2},m_{3},m_{4})$, it suffices that $\mathcal{E}$ is a tree. In Figure \ref{tikz-E1-4} we have several simple possibilities of trees with four branch points. In these tree diagrams, the black nodes are cone points and the white nodes are regular points. The black nodes are labeled 
$j \leftrightarrow z_j$. Typically, there is only one white node or none at all.

\begin{figure}[h!]
\begin{center}
\begin{tabular}
[c]{|c|c|}\hline

\begin{tikzpicture}[scale=0.5, inner sep=1mm]

\node (A) at (-3,0) [shape=circle, fill=black] [label=$1$] {};
\node (B) at (-1,0) [shape=circle, fill=black] [label=$2$]{};
\node (C) at (1,0) [shape=circle, fill=black] [label=$3$]{};
\node (D) at (3,0) [shape=circle, fill=black] [label=$4$]{};
\node (P) at (0,-1) [shape=circle, fill=white] {}; 

\draw [thick] (A) to (B);
\draw [thick] (B) to (C);
\draw [thick] (C) to (D);

\end{tikzpicture}

&
\begin{tikzpicture}[scale=0.5, inner sep=1mm]

\node (A) at (-2,1) [shape=circle, fill=black] [label=$1$]{};
\node (B) at (-2,-1) [shape=circle, fill=black][label=$3$] {};
\node (C) at (0,0) [shape=circle, fill=black] [label=$4$]{};
\node (D) at (2,0) [shape=circle, fill=black] [label=$2$]{};
\node (P) at (0,1.2) [shape=circle, fill=white] {}; 

\draw [thick] (A) to (C);
\draw [thick] (B) to (C);
\draw [thick] (C) to (D);

\end{tikzpicture}
\\ \hline

Cut System: $\mathcal{E}_1$
&
Cut System: $\mathcal{E}_2$
\\\hline

\begin{tikzpicture}[scale=0.5, inner sep=1mm]

\node (A) at (-2,1) [shape=circle, fill=black] [label=$1$] {};
\node (B) at (-2,-1) [shape=circle, fill=black]  [label=$4$] {};
\node (C) at (2,0) [shape=circle, fill=black]  [label=$2$] {};
\node (D) at (4,0) [shape=circle, fill=black]  [label=$3$] {};
\node (W) at (0,0) [shape=circle, fill=white] {}; 
\node (P) at (0,-2) [shape=circle, fill=white]  {}; 

\draw [thick] (A) to (W);
\draw [thick] (B) to (W);
\draw [thick] (W) to (C);
\draw [thick] (C) to (D);
\filldraw[color=black, fill=white, thick] (W) circle (0.3);

\end{tikzpicture}

&
\begin{tikzpicture}[scale=0.5, inner sep=1mm]

\node (A) at (-2,-2) [shape=circle, fill=black]  [label=$4$] {};
\node (B) at (-2,2) [shape=circle, fill=black]  [label=$1$] {};
\node (C) at (2,2) [shape=circle, fill=black]  [label=$2$] {};
\node (D) at (2,-2) [shape=circle, fill=black]  [label=$3$] {};
\node (W) at (0,0) [shape=circle, fill=white] {};
\node (P) at (0,2.2) [shape=circle, fill=white] {}; 

\draw [thick] (W) to (A);
\draw [thick] (W) to (B);
\draw [thick] (W) to (C);
\draw [thick] (W) to (D);
\filldraw[color=black, fill=white, thick] (W) circle (0.3);

\end{tikzpicture}

\\\hline
Cut System: $\mathcal{E}_3$
&
Cut System: $\mathcal{E}_4$
\\\hline
\end{tabular}
\end{center}
\caption{Various cut systems $\mathcal{E}_i$ on the sphere}
\label{tikz-E1-4}
\end{figure}
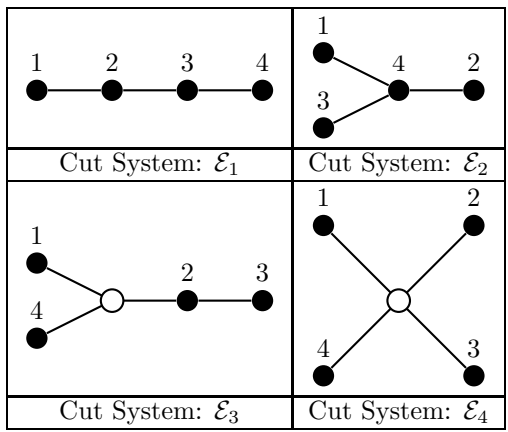

\paragraph{Tiling polygons.}\label{subsubsec-tilepoly}
Now we consider the polygons $\mathcal{P}_i$ polygons in Figure \ref{tikz-P1-4}, derived from the cut systems $\mathcal{E}_i$  in Figure \ref{tikz-E1-4}.  The polygons conceivably have a vertex collapse, we have only pictured the idealized case. In the sphere quotient case, a simple way to construct a plane topological model of the polygon $\mathcal{P}$, given a cut system $\mathcal{E}$, is to do the following. 
\begin{enumerate}
  \item Construct a suitable $2k$-gon in the plane, where $k$ is the number of arcs in the cut system. In the sphere quotient cases this is one less than the number of nodes.   
  \item Construct a ribbon graph neigbourhood $\mathcal{R}_\epsilon$ of the cut system $\mathcal{E}$, as in Figure \ref{tikz-ribbon} of Section \ref{subsubsec-ribbon}.
  \item Starting at the first node $z_1$, traverse the boundary of the ribbon graph in the clockwise direction, recording, in order, the colour and labels  of the nodes that we pass around. 
  \item Starting at a selected initial vertex on the polygon, moving in a \textit{counterclockwise} direction, replace the each vertex by a black node, or white node according to the clockwise order on the ribbon graph. Next, in a counterclockwise fashion, starting at the selected initial vertex label the black nodes with the labelling from the clockwise order on the ribbon graph.             
\end{enumerate}

\begin{remark}\label{rk-E2P} Note that the number of times a label occurs on the polygon is the valency of that vertex in the cut system graph. This also applies to the white node if there is one.
\end{remark}
\clearpage
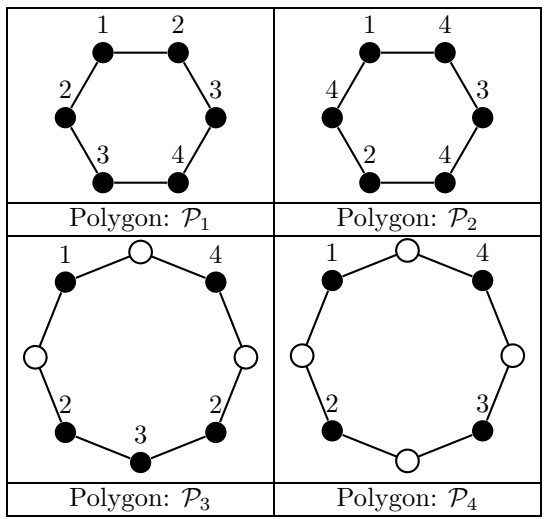
\begin{figure}[h!]
\begin{center}
\begin{tabular}
[c]{|c|c|}\hline

\begin{tikzpicture}[scale=0.5, inner sep=1mm]

\node (A) at (2,0) [shape=circle, fill=black] [label=$3$] {};
\node (B) at (1,1.73) [shape=circle, fill=black] [label=$2$]{};
\node (C) at (-1,1.73) [shape=circle, fill=black] [label=$1$] {};
\node (D) at (-2,0) [shape=circle, fill=black][label=$2$] {};
\node (E) at (-1,-1.73) [shape=circle, fill=black] [label=$3$]{};
\node (F) at (1,-1.73) [shape=circle, fill=black][label=$4$] {};

\draw [thick] (A) to (B);
\draw [thick] (B) to (C);
\draw [thick] (C) to (D);
\draw [thick] (D) to (E);
\draw [thick] (E) to (F);
\draw [thick] (F) to (A);

\end{tikzpicture}

&
\begin{tikzpicture}[scale=0.5, inner sep=1mm]

\node (A) at (2,0) [shape=circle, fill=black] [label=$3$] {};
\node (B) at (1,1.73) [shape=circle, fill=black] [label=$4$]{};
\node (C) at (-1,1.73) [shape=circle, fill=black] [label=$1$] {};
\node (D) at (-2,0) [shape=circle, fill=black][label=$4$] {};
\node (E) at (-1,-1.73) [shape=circle, fill=black] [label=$2$]{};
\node (F) at (1,-1.73) [shape=circle, fill=black][label=$4$] {};

\draw [thick] (A) to (B);
\draw [thick] (B) to (C);
\draw [thick] (C) to (D);
\draw [thick] (D) to (E);
\draw [thick] (E) to (F);
\draw [thick] (F) to (A);

\end{tikzpicture}
\\\hline

Polygon: $\mathcal{P}_1$
&
Polygon: $\mathcal{P}_2$
\\\hline

\begin{tikzpicture}[scale=0.5, inner sep=1mm]

\node (A) at (-2,-2) [shape=circle, fill=black] [label=$2$] {};
\node (B) at (-2,2) [shape=circle, fill=black] [label=$1$]{};
\node (C) at (2,2) [shape=circle, fill=black] [label=$4$] {};
\node (D) at (2,-2) [shape=circle, fill=black][label=$2$] {};

\node (E) at (-2.8,0) [shape=circle, fill=black] {};
\node (F) at (0,2.8) [shape=circle, fill=black] {};
\node (G) at (2.8,0) [shape=circle, fill=black] {};
\node (H) at (0,-2.8) [shape=circle, fill=black][label=$3$] {};

\draw [thick] (A) to (E);
\draw [thick] (E) to (B);
\draw [thick] (B) to (F);
\draw [thick] (F) to (C);
\draw [thick] (C) to (G);
\draw [thick] (G) to (D);
\draw [thick] (D) to (H);
\draw [thick] (H) to (A);

\filldraw[color=black, fill=white, thick](-2.8,0) circle (0.3);
\filldraw[color=black, fill=white, thick](2.8,0) circle (0.3);
\filldraw[color=black, fill=white, thick](0,2.8) circle (0.3);

\end{tikzpicture}

&
\begin{tikzpicture}[scale=0.5, inner sep=1mm]

\node (A) at (-2,-2) [shape=circle, fill=black] [label=$2$] {};
\node (B) at (-2,2) [shape=circle, fill=black] [label=$1$]  {};
\node (C) at (2,2) [shape=circle, fill=black] [label=$4$] {};
\node (D) at (2,-2) [shape=circle, fill=black] [label=$3$] {};

\node (E) at (-2.8,0) [shape=circle, fill=white] {};
\node (F) at (0,2.8) [shape=circle, fill=white] {};
\node (G) at (2.8,0) [shape=circle, fill=white] {};
\node (H) at (0,-2.8) [shape=circle, fill=white] {};

\draw [thick] (A) to (E);
\draw [thick] (E) to (B);
\draw [thick] (B) to (F);
\draw [thick] (F) to (C);
\draw [thick] (C) to (G);
\draw [thick] (G) to (D);
\draw [thick] (D) to (H);
\draw [thick] (H) to (A);

\filldraw[color=black, fill=white, thick](-2.8,0) circle (0.3);
\filldraw[color=black, fill=white, thick](0,-2.8) circle (0.3);
\filldraw[color=black, fill=white, thick](2.8,0) circle (0.3);
\filldraw[color=black, fill=white, thick](0,2.8) circle (0.3);

\end{tikzpicture}

\\\hline
Polygon: $\mathcal{P}_3$
&
Polygon: $\mathcal{P}_4$
\\\hline
\end{tabular}
\end{center}
\caption{Polygons $\mathcal{P}_i$ corresponding to $\mathcal{E}_i$ }
\label{tikz-P1-4}
\end{figure}

\subsection{Crossover calculations}\label{subsec-Xcalcs}
\subsubsection{Crossover calculation setup}\label{subsubsec-crossover2}
We are going the use setup described in Section \ref{subsubsec-crossover1}. 
We need to be able to write the loops in the quotient Cayley graph $\mathcal{C}$ in terms of the generators of $\Gamma$. 

To make this a bit simpler visually, we replace the base point,  $z_0$, by a small disc $D$ surrounding the base point. 
We will assume that $z_0=\infty$ and that all cone points are finite,  once $T$ has been identified with the Riemann sphere $\widehat{\mathbb{C}}$. 
In our rendering, $\partial D$, which is the boundary a small disc about $\infty$ will be a large circle in the plane centered at the origin. 
We also assume the large circle contains the cut system $\mathcal{E}$.
See Figures \ref{tikz-XTloops1} and \ref{tikz-XTloops2} for examples, and to illustrate the steps following.

Let  $\gamma_{1},\gamma_{2},\gamma_{3},\gamma_{4}$ be a standard generating set for $\Gamma$.
We ``groom'' the generators and make $D$ sufficiently small so that each initial ray $\gamma_j^+$ meets $\partial{D} $ in a single point
$\gamma_j^+ \cap \partial{D}$, and  similarly for the terminal ray $\gamma_j^-$.
In Figures \ref{tikz-XTloops1} and \ref{tikz-XTloops2} the $\gamma_j$ can be identified by noting they make a tight turn about the node $z_j$. 
The loops are coloured blue, unless they have been chosen for a crossover loop $\delta$, in which case they are coloured red.

As we traverse $\partial{D}$ in a \emph{counterclockwise} direction, considering $\partial{D}$ as the boundary of $D$, we encounter the ray intersections in this order, according to the standard construction of the $\gamma_j$:
\begin{equation}\label{eq-intersectionordering}
 \gamma_1^+ \cap \partial{D},\gamma_1^- \cap \partial{D},\ldots, \gamma_4^+ \cap \partial{D},\gamma_4^- \cap \partial{D}. 
\end{equation}
When $\partial{D}$ is traversed a \emph{clockwise }direction, considering  $\partial{D}$ as the boundary of $\widehat{\mathbb{C}}-D^\circ$, we get the same order of intersections.  We have also labelled the sides of the arcs in $\mathcal{E}$ with $+$ or $-$ as we did with the ribbon graph.  

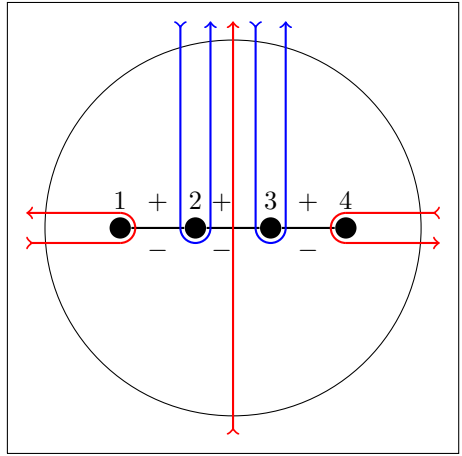
\begin{figure}[h!]

\begin{center}
\begin{tikzpicture}[scale=0.5, inner sep=1mm]

\node (A) at (-3,0) [shape=circle, fill=black] [label=$1$] {};
\node (B) at (-1,0) [shape=circle, fill=black] [label=$2$]{};
\node (C) at (1,0) [shape=circle, fill=black] [label=$3$]{};
\node (D) at (3,0) [shape=circle, fill=black] [label=$4$]{};
\node (P) at (0,-1) [shape=circle, fill=white] {}; 
\node (E1+) at (-2,0.025) [label=$+$] {};
\node (E2+) at (-0.3,0.025) [label=$+$] {};
\node (E3+) at (2,0.025) [label=$+$] {};
\node (E1-) at (-2,-1.25) [label=$-$] {};
\node (E2-) at (-0.3,-1.25) [label=$-$] {};
\node (E3-) at (2,-1.25) [label=$-$] {};

\draw [thick] (A) to (B);
\draw [thick] (B) to (C);
\draw [thick] (C) to (D);

\draw[color=black,thin] (0,0) circle (5);
\draw [color=black, thin] (-6,-6) to (6,-6) to (6,6) to (-6,6) to (-6,-6) ;

\draw [color=red, thick,<-] (-5.5,0.4) to (-3,0.4);
\draw [color=red, thick, >-] (-5.5,-0.4) to (-3,-0.4);
\draw [color=red, thick] (-3,0.4) arc[start angle=90, delta angle = -180, radius = 0.4];

\draw [color=blue, thick,>-] (-1.4,5.5) to (-1.4,0);
\draw [color=blue, thick,<-] (-0.6,5.5) to (-0.6,0);
\draw [color=blue, thick] (-0.6,0) arc[start angle=0, delta angle = -180, radius = 0.4];

\draw [color=red, thick,>->] (0,-5.5) to (0,5.5);

\draw [color=blue, thick,>-] (0.6,5.5) to (0.6,0);
\draw [color=blue, thick,<-] (1.4,5.5) to (1.4,0);
\draw [color=blue, thick] (0.6,0) arc[start angle=180, delta angle = 180, radius = 0.4];

\draw [color=red, thick,>-] (5.5,0.4) to (3,0.4);
\draw [color=red, thick, <-] (5.5,-0.4) to (3,-0.4);
\draw [color=red, thick] (3,0.4) arc[start angle=90, delta angle = 180, radius = 0.4];

\end{tikzpicture}
\end{center}

\caption{Crossover loops (red) for $\mathcal{E}_1$ with orientation markings}
\label{tikz-XTloops1}

\end{figure}

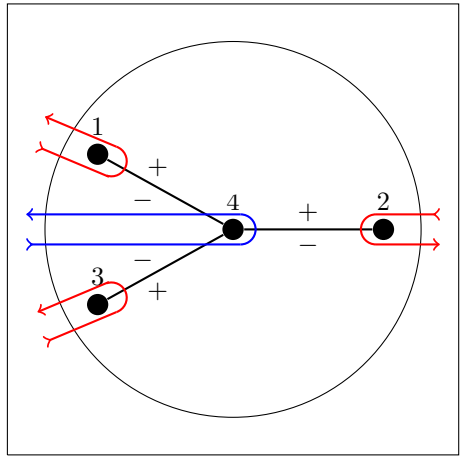
\begin{figure}[h!]

\begin{center}
\begin{tikzpicture}[scale=0.5, inner sep=1mm]

\node (A) at (-3.6,2) [shape=circle, fill=black] [label=$1$]{};
\node (B) at (-3.6,-2) [shape=circle, fill=black][label=$3$] {};
\node (C) at (0,0) [shape=circle, fill=black] [label=$4$]{};
\node (D) at (4,0) [shape=circle, fill=black] [label=$2$]{};
\node (P) at (0,1.2) [shape=circle, fill=white] {}; 
\node (E1+) at (-2,1) [label=$+$] {};
\node (E2+) at (2,-0.2) [label=$+$] {};
\node (E3+) at (-2,-2.3) [label=$+$] {};
\node (E1-) at (-2.4,0.1) [label=$-$] {};
\node (E2-) at (2,-1.1)[label=$-$] {};
\node (E3-) at (-2.4,-1.5) [label=$-$] {};

\draw [thick] (A) to (C);
\draw [thick] (B) to (C);
\draw [thick] (C) to (D);

\draw[color=black,thin] (0,0) circle (5);
\draw [color=black, thin] (-6,-6) to (6,-6) to (6,6) to (-6,6) to (-6,-6) ;

\draw [color=blue, thick,<-] (-5.5,0.4) to (0.2,0.4);
\draw [color=blue, thick, >-] (-5.5,-0.4) to (0.2,-0.4);
\draw [color=blue, thick] (0.2,0.4) arc[start angle=90, delta angle = -180, radius = 0.4];

\draw [color=red, thick,<-] (-5,3) to (-3.1,2.2);
\draw [color=red, thick,>-] (-5.2,2.2) to (-3.3,1.4);
\draw [color=red, thick] (-3.1,2.2) arc[start angle=71, delta angle = -180, radius = 0.405];

\draw [color=red, thick,>-] (-5,-3) to (-3.1,-2.2);
\draw [color=red, thick,<-] (-5.2,-2.2) to (-3.3,-1.4);
\draw [color=red, thick] (-3.1,-2.2) arc[start angle=-71, delta angle = 180, radius = 0.405];

\draw [color=red, thick,>-] (5.5,0.4) to (3.8,0.4) ;
\draw [color=red, thick, <-](5.5,-0.4) to (3.8,-0.4) ;
\draw [color=red, thick] (3.8,0.4) arc[start angle=90, delta angle = 180, radius = 0.4];

\end{tikzpicture}
\end{center}

\caption{Crossover loops (red) for $\mathcal{E}_2$ with orientation markings}
\label{tikz-XTloops2}

\end{figure}

\subsubsection{Detailed crossover calculations for $\mathcal{E}_1$}\label{subsec-E1detail}
Using Figure \ref{tikz-XTloops1}, we can determine crossover transformations by orienting and lifting three loops  $\delta_1,\delta_2, \delta_3$ that we choose for $\mathcal{C}_1$.  The chosen crossover loops, plotted in red, are:
\[ 
\delta_1=\gamma_1,\  \delta_2 = \mathrm{ vertical\ red\ loop,\ and\ } \delta_3=\gamma_4.
\]

Using the big circle we can easily visualize homotopies. 
For instance $\gamma_1^- \cap \partial{D}$ and $\gamma_2^+ \cap \partial{D}$
have no intersection point between them and can be squeezed together, 
joined, and then the new intersection point moved down the diagram 
through the arc between $z_1$ and $z_2$, The new loop encircles both $z_1$ and $z_2$, 
This move corresponds to  moving the snippet $\gamma_1^-\gamma_2^+$ away from  $z_0$ 
and down the diagram, such that the homotopy extends to all of $\gamma_1\gamma_2$. 
Next, we can do a similar sliding operation on the new loop to match it up with $\delta_2$. So we have: 
\[ 
\delta_1=\gamma_1,\  \delta_2 = \gamma_1\gamma_2,\ \delta_3=\gamma_4
\]
and, using the generating vector  $(c_1,c_2,c_3,c_4)$,

 \begin{eqnarray*}
 \xi(\delta_1) &=& c_1 \\  
 \xi(\delta_2) &=& c_1c_2 \\
 \xi(\delta_3) &=& c_4
 \end{eqnarray*}

Our last step is to determine  quantities $\xi(\delta^+_i)$ and  $\xi(\delta^-_i)$, 
$i = 1,2,3$. Using the method in Section \ref{subsubsec-e-orientation} to determine orientation, 
we compute all the crossover transformations in Table 1. 
Let us explain the notation in the columns. In the first column $\tilde{e}_i$ on the 
left hand side of $\leftrightarrow$ is the open edge in $\mathcal{P}_1$  
in the order we encounter them in a counterclockwise circuit around $\partial\mathcal{P}_1$. 
On the right hand side of $\leftrightarrow$ the symbol $e_j^\pm$ 
indicates which edge, and which orientation is associated with $\tilde{e}_i$ via $\pi_G$. 
Now the subscripts indicate which edge in $\mathcal{E}$ is being selected, the numbering 
going $1,2,3$ as we sweep left to right in the diagram for $\mathcal{E}_1$. 
The middle columns give a formula for the oriented crossover loops as elements of $\Gamma$. 
The last column gives us the crossover transformation $\tau_{\tilde{e}} \in G$.

 \[%
 \label{table-4}
\begin{tabular}
[c]{c}%
\begin{tabular}
[c]{|l|l|l|}\hline
Edge $e$ & $\delta \in \mathcal{C}$ & $\tau_{\tilde{e}}$\\ \hline
$\tilde{e}_1\leftrightarrow e^+_1$ & $\delta^+_1=\gamma_1^{-1}$ & $c_1^{-1}$\\\hline
$\tilde{e}_2\leftrightarrow e^+_2$& $\delta^+_2=(\gamma_2\gamma_3)^{-1}$ &  $(c_1c_2)^{-1}$ \\\hline
$\tilde{e}_3\leftrightarrow e^+_3$ & $\delta^+_3=\gamma_4$ &  $c_4$ \\ \hline
$\tilde{e}_4\leftrightarrow e^-_3$ & $\delta_3^-=\gamma_4^{-1}$ &  $c_4^{-1}$\\ \hline
$\tilde{e}_5\leftrightarrow e^-_2$ & $\delta_2^-=\gamma_2\gamma_3$ & $c_1c_2$\\ \hline
$\tilde{e}_6\leftrightarrow e^-_1$& $\delta_1^-= \gamma_1$ & $c_1$\\ \hline
\end{tabular}
\ \\
\\
Table 1. Crossover transformations for $\mathcal{P}_1$
\end{tabular}
\
\]

\subsubsection{Crossover table for various cut systems}\label{subsubsec-TableSQ}
We finish up crossover calculations for the sphere quotient case with Table 3, showing the sequences
$e_j^\pm=\pi_G(\tilde{e}_i),\text{\ } i=1,\ldots,2k$ and the associated crossover transformations $\tau_{\tilde{e}_i}, i=1,\ldots,2k$. 
The $\tilde{e}_i$ are ordered by a counterclockwise circuit about a distinguished polygon (Figure \ref{tikz-P1-4})  
or, alternatively, a clockwise circuit around the ribbon graph neighbourhood 
(see examples in Figures \ref{tikz-ribbon}, \ref{tikz-XTloops1}, \ref{tikz-XTloops2}). 
To aid in associating the edges in the polygon boundary and those in the cut system, 
we note that the undirected edges $e_j$ in $\mathcal{E}$ can be determined by the set of endpoints in the cut system.  
In Table 2, where undirected edges are defined, we use the node number for the black nodes and $w$ for the white node.

\[\label{table-5}
\begin{tabular}[c]{c}
\begin{tabular}
[c]{|l|l|l|}\hline
$\mathcal{E}$    & $e_j$ undirected edges \\ \hline
$\mathcal{E}_1$  & $e_1=\left\{1,2\right\}$, $e_2=\left\{2,3\right\}$, $e_3=\left\{3,4\right\}$ \\ \hline
$\mathcal{E}_2$  & $e_1=\left\{1,4\right\}$, $e_2=\left\{2,4\right\}$, $e_3=\left\{3,4\right\}$  \\ \hline         
$\mathcal{E}_3$  & $e_1=\left\{1,w\right\}$, $e_2=\left\{2,w\right\}$, $e_3=\left\{2,3\right\}$, $e_4=\left\{4,w\right\}$ \\ \hline
$\mathcal{E}_4$  & $e_1=\left\{1,w\right\}$, $e_2=\left\{2,w\right\}$, $e_3=\left\{3,w\right\}$, $e_4=\left\{4,w\right\}$ \\ \hline
\end{tabular}
\ \\
\\
Table 2. Undirected edges in the cut systems
\end{tabular}
\]

To aid in interpreting and constructing the crossover table, we may utilize the cut system diagrams  
(Figure \ref{tikz-E1-4}), the polygon diagrams (Figure \ref{tikz-P1-4}), 
the analogues of the sample ribbon graph neighbourhood in Figure \ref{tikz-ribbon}, and Figures \ref{tikz-XTloops1}, \ref{tikz-XTloops2}. 

\[\label{table-6}
\begin{tabular}[c]{c}
\begin{tabular}
[c]{|l|l|l|}\hline
$\mathcal{E}$    & $e$: oriented edge sequence 
                 & $\tau_{\tilde{e}}$ sequence \\ \hline
$\mathcal{E}_1$  & $e_1^+$, $e_2^+$, $e_3^+$, $e_3^-$, $e_2^-$, $e_1^-$
                 & $c_1^{-1}$, $(c_1c_2)^{-1}$, $c_4, c_4^{-1}, c_1c_2, c_1$ \\ \hline
$\mathcal{E}_2$  & $e_1^+$, $e_2^+$, $e_2^-$, $e_3^+$, $e_3^-$, $e_1^-$ 
                 & $c_1^{-1}$, $c_2$, $c_2^{-1}, c_3, c_3^{-1}, c_1$ \\ \hline           
$\mathcal{E}_3$  & $e_1^+$, $e_2^+$, $e_3^+$, $e_3^-$, $e_2^-$, $e_4^+$, $e_4^-$,  $e_1^-$
                 & $c_1^{-1}$, $(c_4c_1)^{-1}$, $c_3$, $c_3^{-1}$, $c_4c_1$, $c_4$, $c_4^{-1}$, $c_1$  \\ \hline
$\mathcal{E}_4$  & $e_1^+$, $e_2^+$, $e_2^-$, $e_3^+$, $e_3^-$, $e_4^+$, $e_4^-$,  $e_1^-$
                 & $c_1^{-1}$, $c_2$, $c_2^{-1}$, $c_3$, $c_3^{-1}$, $c_4$, $c_4^{-1}$, $c_1$   \\ \hline
                 
\end{tabular}
\ \\
\\
Table 3. Crossover transformation sequences for various cut systems
\end{tabular}
\]

\begin{remark}
Some observations:
\begin{itemize}
  \item The cut system $\mathcal{E}_1$ works well when analyzing (anti-conformal) symmetries of surfaces, or other actions where the branch points are all real.
  \item The cut systems $\mathcal{E}_2$ and $\mathcal{E}_4$ cannot have collapsed edges, as there are no trivial crossover transformations.
  \item The cut systems  $\mathcal{E}_1$ and $\mathcal{E}_3$ have collapsed edges if and only if $c_2=c_1^{-1}$ and  $c_4=c_1^{-1}$, respectively.   
  \item For an arbitrary number of branch points on a sphere quotient, the analogue of $\mathcal{E}_4$ will have no collapsed edges.
\end{itemize}
\end{remark}

\subsection{Sector labelling sequences}\label{subsec-sectorlabrling}
We illustrate the possibilities of sector labelling sequences by computing the sequences 
for the nodes in Figures \ref{tikz-XTloops1} and \ref{tikz-XTloops2}.

First we consider the trivial case of a vertex that is a terminal node of $\mathcal{E}$. 
By calculation, the sector labelling sequence is $h, hc, hc^2,\ldots,hc^{m-1}$,
where  $h$ is the initial label, $c=c_j$ or $c=c_j^{-1}$, and $m=m_j$, as determined by the generating vector. 
This applies to any terminal node in any cut system.  

Now consider the node number $2$ in Figure \ref{tikz-XTloops1}. If we make a small counterclockwise loop
about the vertex, starting in the upper half plane, we encounter the oriented edges $e_1^+$ and $e_2^-$, using the notation in Tables 2 and 3.
From the tables we get $\tau_{e_1^+}=c_1^{-1}$,  $\tau_{e_2^-}=c_1c_2^{-1}$. 
Noting that the third element of the sequence is $h(c_1^{-1})(c_1c_2)=hc_2$, the sequence is then:
\[
h, hc_1^{-1}, hc_2, hc_2c_1^{-1}, hc_2^2, hc_2^2c_1^{-1},\ldots, hc_2^{m_2-1}, hc_2^{m_2-1}c_1^{-1}.  
\]
We can easily see the form of the pattern induced by the stabilizer of the vertex. In a specific case we can check to see if this sequence has repetitions. The calculation for node 3 is similar.

\begin{remark}\label{rk-no-h}
Note that the initial label does not change the pattern of sector transitions, or of vertex collapses, if there are any. So we can safely set $h=1$.
\end{remark}

For a more complex example, we consider node 4 of valence 3 in Figure \ref{tikz-XTloops2}. 
Starting in the sector with the blue loop, the relevant crossover transformations, in order, are:
\[
\tau_{e_3^-}=c_3^{-1}, \text{\ } \tau_{e_2^-}=c_2^{-1}, \text{\ } \tau_{e_1^+}=c_1^{-1},
\]
with product
\[
\tau_{e_3^-}\tau_{e_2^-}\tau_{e_1^+} = c_3^{-1}c_2^{-1}c_1^{-1}=c_4.
\]
The sequence is
\[
1, c_3^{-1}, c_3^{-1}c_2^{-1}, c_4, c_4c_3^{-1}, c_4c_3^{-1}c_2^{-1}, \ldots, c_4^{m_4-1}, c_4^{m_4-1}c_3^{-1}, c_4^{m_4-1}c_3^{-1}c_2^{-1}.
\]

We can use the following proposition as a check on our crossover and sector labeling sequences, valid for any genus of $T$.

\begin{proposition}\label{prop-sectorlabellingproduct} Let all notation be as above, in particular the
monodromy $\xi:\Gamma\rightarrow G$ is fixed for all our calculations.   
For any vertex $v \in \mathcal{E}$ of valency $k$, let $D$ be a small closed disc centered at $v$ 
and containing no other vertices of $\mathcal{E}$. Let $\tau_1,\ldots,\tau_k$ denote the crossover transformations
of the sector boundaries (spokes) as we make a counterclockwise circuit along $\partial D$, starting in a designated sector.
Let $\eta \in \Gamma$ denote the loop based at $z_0$; joining the circle $\partial D$ at a point in the designated sector; 
traversing $\partial D$ once, in a counterclockwise direction; and then returning to the basepoint along the initial path.
Then 
\begin{equation}\label{eq-taucycleeqn}
\tau_1\cdots\tau_k=\xi(\eta).
\end{equation}
If $v$ is a white vertex, then the product is trivial, no matter what sector we start in.
\end{proposition}

\begin{proof}
Let $\delta_1,\ldots,\delta_k$ be a sequence of crossover loops, defining the crossover sequence 
by $\tau_i=\xi(\delta_i)$. Each $\delta_i$ may be constructed so that it crosses its associated spoke 
on $\partial D$,  in the counterclockwise direction.  For every product $\delta_i\delta_{i+1}$ there are 
three portions of the combined loop: the portion of $\delta_i$ before it hits its associated spoke, 
followed by the portion of $\delta_i\delta_{i+1}$ before $\delta_{i+1}$ hits its associated spoke, 
and then the remaining portion of $\delta_{i+1}$. Since the middle portion lies entirely 
within the open polygon, except the end points, we may construct a homotopy that moves the middle portion of 
$\delta_i\delta_{i+1}$ to a path along $\partial D$ that remains within the sector. 
Applying such a homotopy to all sectors, except the initial sector, 
we see that  $\delta_1\cdots\delta_k$ is homotopic to $\eta$. 

The equation  \eqref{eq-taucycleeqn}
then follows from the definitions and  $\delta_1\cdots\delta_k =\eta$. If $v$ is a white node, then $\eta$ 
is trivial and $\tau_1\cdots\tau_k=\xi(\eta)=1$.  We may permute the factors in the crossover product cyclically:
\[
1= \tau_1^{-1}\tau_1\cdots\tau_k\tau_1 = \tau_2\cdots\tau_k\tau_1
\] 
Iterating this shift, we see that we may start the product in any sector, still with a product of 1. 
\end{proof}

\section{Cayley graphs, modular companions, and partial isometries}\label{sec-Cayley} 
In constructing the surface $S$ from a tiling, the polygons $\left\{  g\mathcal{P}:g\in G\right\}$ are the puzzle pieces 
alluded to previously. The monodromy $\xi$ and the derived crossover transformation sequence $XTS$ are ``recipes'' telling us how 
to conformally glue labelled polygons $\left\{  g\mathcal{P}:g\in G\right\}  $ together to get
a surface with $G$ action. If $\xi^{\prime}=\xi\circ\mu^{-1}$ is another
monodromy we have a modified instruction set to piece together $\left\vert
G\right\vert $ copies of the same polygon.  So, some of the geometry of the
polygon $\mathcal{P}$ is retained in all of the modular companions. A third type of \textquotedblleft monodromy recipe\textquotedblright, 
with much more detail and with both geometric and algebraic interpretations, is the \emph{dual, embedded Cayley graph}. 
Our final main section is devoted to this topic and its application to the study of modular companions, and partial isometries.

\begin{remark}\label{rk-any2surfs}
Note that we are focusing on comparisons of modular companions, but in fact any two surfaces 
lying over $T$ with appropriate monodromies can be compared, including partial isometries. 
\end{remark}

\subsection{The dual Cayley graph}\label{subsec-CaleyGraph}
The embedded graph $\widetilde{\mathcal{E}}$ in $S$ defines a map on $S$ and so has an embedded dual graph
$\widetilde{\mathcal{C}}$, which also defines a map on $S$.
We call $\widetilde{\mathcal{C}}$ the \emph{dual Cayley graph}, 
and note that it can be made $G$ invariant.
Abstractly, as a geometrically constructed graph, the vertices of $\widetilde{\mathcal{C}}$ are
the lifted polygons $\left\{  g\mathcal{P}:g\in G\right\}$. Two vertices are
connected by a bi-directional edge in $\widetilde{\mathcal{C}}$ if and only the polygons meet along an edge in
$\widetilde{\mathcal{E}}$, producing one edge in $\widetilde{\mathcal{C}}$ for each edge in the intersection. 
The faces of $S-\widetilde{\mathcal{C}}$ correspond to the vertices of $\widetilde{\mathcal{E}}$. These faces may be constructed with the help of the sector labelling sequences discussed in Section \ref{subsubsec-sectorlabelling}. 

In order to get a good relationship to the various group Cayley graphs that we discuss in Section \ref{subsec-groupCayley}, 
we impose the following restriction.
\begin{assumption}
Let $\mathcal{E} \subseteq T$ and  $\widetilde{\mathcal{E}} \subseteq S$ be as previously defined. 
When constructing a dual Cayley graph we assume that there are no edge collapses in the tiling on $S$, determined by 
$\widetilde{\mathcal{E}}$. 
This is equivalent to having no trivial crossover  transformations $\tau_{e^\pm}$ for edges  $e \in \mathcal{E}$, 
or alternatively, no loops in the dual Cayley graph. 
\end{assumption}

Here is our definition of the dual Cayley graph, which also gives us a construction method.

\begin{definition}\label{def-cayley dual}
Let  $G$ act conformally on  $S$, with quotient orbifold $T=S/G$, and let $\mathcal{E}$ be a cut system on $T$. 
Suppose that $\mathcal{C}= \{\delta_e: e \in \mathcal{E}\}$ is a set of crossover loops that meet only at their base point $z_0\in T^\circ$. 
Assume that each loop in $\mathcal{C}$ crosses, transversally, exactly one arc from $\mathcal{E}$,
and that the intersection point is in the interior of both the loop and the arc. 
Such a system of loops on $T$ will be called the \emph{dual Cayley graph} of  $\mathcal{E}$. 
The lift $\widetilde{\mathcal{C}}$ of $\mathcal{C}$ to $S$ will be called the \emph{dual Cayley graph} of  $\widetilde{\mathcal{E}}$.
\end{definition} 

\begin{remark}\label{rk-Cayleymap}
If there are no edge collapses, then each edge of $\widetilde{\mathcal{C}}$ has distinct endpoints since they lie in disjoint polygons, and so edges cannot be loops. Also, if we assign different colours to the loops in  $\mathcal{C}$, we obtain a lifted colouring on $\widetilde{\mathcal{C}}$. 
\end{remark}

\begin{example}\label{ex-S3-2233-3}  
For examples with 4 cone points, we describe a picture of a local model of the graph $\widetilde{\mathcal{C}}$ 
superimposed upon a model local picture of $\widetilde{\mathcal{E}}$. Having chosen $\mathcal{E}$ from Figure \ref{tikz-E1-4}
we may draw upon the corresponding polygon from Figure \ref{tikz-P1-4}, which we assume corresponds to the Cayley group element node $1$.
Draw a series of spokes from the center of the polygon to slightly beyond the midpoint of each side.
At the end of each spoke write down the group element of the new polygon entered while traveling outwards along the spoke. 
Of course this is just the crossover sequence $\tau_{\tilde{e}_1},\ldots,\tau_{\tilde{e}_{2k}}$,
written down in counterclockwise-order, starting at the appropriate edge. 

Now let us specifically pick $\widetilde{\mathcal{E}_4}$ and $\widetilde{\mathcal{P}_4}$, and the two modular companions
determined by $\mathcal{V}_1,\mathcal{V}_2$ in Example \ref{ex-S3-2233-2}. According to Table 3  we get these crossover sequences.
\begin{eqnarray*}
 \mathcal{V}_1=(x,x,y,y^{-1}) &\rightarrow& (x,x,x,y,y^{-1},y^{-1},y,x),\\
 \mathcal{V}_2=(x,xy,y,y)     &\rightarrow& (x,xy,xy,y,y^{-1},y, y^{-1},x).  
\end{eqnarray*}
For the local picture at the node labeled $g$ we simply left translate all the labels by $g$. 
Clearly, the multi-edge structure of the two graphs are different.  See also Section \ref{subsec-modcompdiff}.
\end{example}

\subsection{Group Cayley graphs}\label{subsec-groupCayley}
The obvious algebraic counterpart to the dual Cayley graph $\widetilde{\mathcal{C}}$ is the standard group Cayley graph, 
$\mathrm{Cay}(C,\Sigma)$, described in the following paragraphs (\ref{subsubsec-groupCayley1}).  It turns out that the standard group Cayley graph is 
insufficient to capture all the details of our tilings because of multi-edges, see Proposition \ref{prop-multiedge}. So, we introduce a modified version in Section \ref{subsubsec-groupCayley2}.   

\subsubsection{Standard group Cayley graph}\label{subsubsec-groupCayley1}
The standard \emph{group Cayley graph} $\mathrm{Cay}(C,\Sigma)$ is described in the points below. The generating set is $\Sigma=\left\{\tau_e^\pm, e \in \mathcal{E}\right\}$,  
which acquires two generators  $\tau_e^\pm$,  one for each orientation of $e$, except when $\tau^2_e=1$. 
This double counting ensures that  $\Sigma$ is a symmetric generating set.

The simplest version of a \emph{symmetric, connected, undirected group Cayley graph}, $\mathrm{Cay}(G,\Sigma)$, for $\Sigma \subseteq G$ is characterized by:

\begin{enumerate}
  \item Nodes: $\{g \in G\}$;
  \item Edges: $\{(g,g\sigma) , g \in G, \sigma \in \Sigma\}$, \\
   edges determined by endpoints, so no multi-edges;
  \item Connected: $\Sigma$ generates $G$;
  \item Symmetry: $\sigma^{-1} \in \Sigma \text{ if } \sigma \in \Sigma$ (edges are bidirectional); and
  \item No loops: $\Sigma$ does not contain the identity of  $G$.
\end{enumerate}  

The sole issue that may cause the dual Cayley graph and the associated group Cayley graph to be different is allowing for repeated generators. 
In $S$, two distinct polygons $g\mathcal{P}$ and  $h\mathcal{P}$  may meet in more than one edge in $\widetilde{\mathcal{E}}$, 
and so there may be multiple bidirectional edges between vertices of  $\widetilde{\mathcal{C}}$. 
This situation is described completely in terms of repeated generators in Proposition \ref{prop-multiedge}. 
On the other hand, as noted in item 2 above, edges are determined by their endpoints in the standard group Cayley graph,
so there cannot be any multi-edges. 
We can adjust our definition of a group Cayley graph to repair this discrepancy, see Section \ref{subsubsec-groupCayley2}.
       
Other related graphs are the group Cayley graph determined by 
$(G,\mathcal{V})$, where $\mathcal{V}$ is a generating vector and 
the \emph{Reidemeister-Schreier coset graph} determined by the pair $(\Gamma,\Pi)$ in equation \eqref{eq-Fuchspair}
with the generating set $\mathcal{W}$ in equation \eqref{eq-W2V}. 
Of course we should symmetrize the generating sets. These graphs and Fuchsian pairs are discussed \cite{Si1,Si2}.

\subsubsection{Modified group Cayley graph}\label{subsubsec-groupCayley2} 
To address the issue caused by multi-edges, we define the \emph{modified group Cayley graph},  $\mathrm{Cay}(G,E,\Sigma)$,
where $E$ denotes the set of directed edge types, and $\Sigma=\left\{\tau_e^\pm, e \in \mathcal{E}\right\}$ is the generating set,
The colour of an edge of type $e$ and its opposite $e^{op}$ are the same. The nodes of  $\mathrm{Cay}(G,E,\Sigma)$ are still the group elements. 
The directed edges are triples $(e,g,g\tau_e)$, connecting $g$ to $g\tau_e$. The opposite edge, connecting  $g\tau_e$ to $g$, is  
\[
(e^{op},g\tau_e,g)=(e^{op},g\tau_e,g\tau_e\tau_{e^{op}}),
\]
noting that $\tau_{e^{op}}=\tau_e^{-1}$. 
The modified version allows us to have repeat generators, particularly involutions, which are very common. 
The two edges $(e,g,g\tau_e)$ and  $(e^{op},g\tau_e,g)$ correspond to the same arc in the 
dual Cayley graph with opposite orientations. 
To help visualize the graph, different edge types can be drawn according to the colouring scheme.   
It is easily shown that when there are no edge collapses, that the dual Cayley graph 
and the modified group Cayley graph are isomorphic.

\subsection{Cayley graphs and modular companions}\label{subsec-CayletModComp}

As previously promised in Section \ref{subsec-modcomp}, we will use Cayley graphs to geometrically
distinguish two modular companions. We first look at some crude measures that clearly show that the geometry of the tilings and the Cayley graphs
coming from two surfaces lying over $T$ are different. 

For modular companions that cannot be so distinguished we use \emph{partial isometries} to more finely distinguish them. 
We see from the Table 3 and the worked examples of sector labelling sequences, that everything can be computed in 
terms of a generating vector, $\mathcal{V}$, or more specifically the crossover transform sequence
\[
XTS=XTS(\mathcal{V})=(\tau_1,\ldots,\tau_{2k})=(\tau_{\tilde{e}_1},\ldots,\tau_{\tilde{e}_{2k}}). 
\]
The model polygons, including edge type sequence and orientation, are the same for every surface 
lying over $T$, so we can use the model polygons and the crossover transform sequence to detect geometrical differences between these surfaces. 
See also Example \ref{ex-S3-2233-3}.

\subsection{Differentiating features for surfaces over $T$}\label{subsec-modcompdiff}  
Three geometrical features that we can use to differentiate surfaces lying over $T$
are: 1) polygons that have an edge collapse, 2) polygons that meet in multiple edges, and 3) polygons that have a vertex collapse. 
As we noted previously, the characteristics are the same for every polygon in the tiling. 
All the features can be diagrammed on the model polygon and computed with the crossover sequence $(\tau_1,\ldots,\tau_{2k})$.
For the planar actions with four cone points considered in this paper the polygons are very simple (Figure \ref{tikz-P1-4}), 
and the pictures can be sketched by hand.   

First we consider edge collapses.  Take our model polygon and place an edge label next to each edge. 
They should be of the form $e_i^\pm$ (see Table 3), so that we can easily see paired sides.
For any two sides for which there is an edge collapse, draw the sides  in a unique color, other than black.
From the sequence  $(\tau_1,\ldots,\tau_{2k})$ we find the number of collapsed edges 
by counting the number of trivial entries, which will be even. The polygon display shows a bit more detail by showing the placement of the collapsed edge pairs.    
     
Next, we consider multi-edge intersections. Again we use our polygon model with edge labels. 
Some edges will have a crossover transformation which has repetitions in the crossover transformation sequence. Color all the edges with this crossover  transformation with the same, unique color. All of the edges whose crossover transformation has no repeats should be coloured black. Again, the polygon display gives more visual detail than the crossover transformation sequence. See also Example \ref{ex-S3-2233-3}.

Finally, vertex collapse can be a differentiator among modular companions.  These can be visualized in the polygon by drawing an interior diagonal between two vertices that meet at the same point, in $S$.  The vertices must have the same label.
 
\subsection{Partial isometries}\label{subsec-partialIsoCay}
Partial isometries are our main use of Cayley graphs and our main tool for distinguishing modular companions.
Before getting started, lets give a forward reference to some basic facts recorded as Lemmas in the last section of the paper, Section  \ref{sec-BasicLemmas}. We leave them to the end  of the paper so as not to disturb the flow of the narrative.
Also, Remark \ref{rk-any2surfs} applies  here.

Suppose that we want to construct an isometry (biholomorphic map) between two different modular companions
 $S_1$ and $S_2$, that commutes with the orbifold projections  $\pi_1, \pi_2$. 
Though this is doomed to failure, we could try to get a partial isometry $\phi: \Omega_1 \rightarrow \Omega_2$, where  $\Omega_i \subseteq S_i,\text{\ } i=1,2$  are large connected subsets (diagram \eqref{dia-partialiso}). This can give us a measure of how different the two surfaces are combinatorially. 

\begin{equation}\label{dia-partialiso} 
\xymatrix{
    \Omega_1 \ar[r]^\phi \ar[d]_{\pi_1} & \Omega_2 \ar[d]^{\pi_2} \\
    T \ar[r]^{id}       & T }
\end{equation}
In visual terms, the two sets $\Omega_1$, $\Omega_2$ are patches on the surfaces that match isometrically when we try to align one surface with the other.
So, we call $\Omega_1$, $\Omega_2$ \emph{patches}  and $\phi$ a \emph{patch match}.

Now, let $\mathcal{E}$ be a cut system on $T$ and let  $\mathcal{C}$ be the dual Cayley graph, determined by the cut system. 
Each of the modular companions $S_1,\ldots,S_\ell$ lying  over $T$, is determined by a monodromy $\xi_i:\Gamma\rightarrow G$, which in turn 
determines a holomorphic projection $\pi_i: S_i \rightarrow T$.  
Using the projections we can pull up $\mathcal{E}$ and  $\mathcal{C}$ to covering graphs 
$\widetilde{\mathcal{E}}_i$ and $\widetilde{\mathcal{C}}_i$, respectively.
For each $i$, a distinguished open polygon $\widetilde{\mathcal{P}}^\circ_i \subset S_i$ 
may be selected from among the open components of $\pi^{-1}_i(T-\mathcal{E})$.
For convenience we will drop the  $\widetilde{\ }$  over the polygons and let the suffix determine where they live.
Once we have determined the sequence of oriented edges of $\mathcal{E}$, $e_1,\ldots,e_{2k}$, as a polygon boundary, 
the crossover transformations in $S_i$ may be determined by
\[
\tau_{j,i} = \xi_i(\delta_{e_j}),
\]  
for an appropriately oriented loop $\delta_{e_j}$ in $\mathcal{C}$. 
The geometry of the modular companions are completely determined from the fixed geometry on $T$, the crossover sequences 
$XTS_i=(\tau_{1,i},\ldots,\tau_{2k,i})$, and the biholomorphic maps $\pi_i: \mathcal{P}^\circ_i \rightarrow T-\mathcal{E}$.

 We next focus on an arbitrary pair of modular companions, which we assume are $S_1$ and $S_2$. 
A partial isometry always exists, namely
\begin{equation}\label{eq-phizero}
\phi_0=\pi_2^{-1}\circ\pi_1: \mathcal{P}^\circ_1 \rightarrow \mathcal{P}^\circ_2, 
\end{equation} 
since the projection maps are biholomophic when restricted to the distinguished  open polygons.
We can do a little better. Let $g \in G$ and let $g^\prime=w(g)$,  be (for the moment) an arbitrary element of $G$. 
Then we have the following biholomorphic map
\begin{equation}\label{eq-phiggprime}
\phi_{g,g^\prime}=\epsilon_2(g^\prime)\circ\phi_0\circ\epsilon_1(g)^{-1} : g\cdot\mathcal{P}^\circ_1 \rightarrow g^\prime\cdot\mathcal{P}^\circ_2
\end{equation}  
(two different actions!). If $g\rightarrow g^\prime$ is a bijection then $\bigcup_{g \in G} \phi_{g,g^\prime}$ is a biholomorphic map from $S_1-\widetilde{\mathcal{E}}_1$ to  $S_2-\widetilde{\mathcal{E}}_2$. All we have to do is fix it up on the edges, there could be some ripping and tearing! 

\medskip
Next, we extend these maps to nipped polygons.
\begin{lemma}\label{lem-extend2nipped}
Suppose that $\mathcal{P}^\circ_1$ and $\mathcal{P}^\circ_2$ have their edge collapses in exactly the same locations, i.e., $\tau_{j,1} =1$ if and only if 
$\tau_{j,2} =1$. Then the  maps in \eqref{eq-phizero} and \eqref{eq-phiggprime} extend to homeomorphisms
\[
\phi_0=\pi_2^{-1}\circ\pi_1: \mathcal{P}_1^\star \rightarrow \mathcal{P}_2^\star, 
\]
and 
\[
\phi_{g,g^\prime}=  g\cdot\mathcal{P}_1^\star \rightarrow g^\prime\cdot\mathcal{P}_2^\star
\]
Moreover, the maps preserve the type of edges. 
\end{lemma}  

\begin{proof}
We need only prove the lemma for $\phi_0$ as the second statement follows easily from the first. 
The first statement may be proved using Lemma \ref{lem-extension}. Some care is needed, as one must look at 
the case of a simultaneous edge collapse as a special case.    
\end{proof}

In what follows, let  $H=H_1, \text{  }H^\prime=H_2$ be subsets of $G$, and $w:g\rightarrow g^\prime$ be a bijective map as previously introduced  (primes used to simplify notation).
We will consider the partially defined map  $\bigcup_{g \in H} \phi_{g,g^\prime}$. 
Consider these sets:
\begin{eqnarray*}
 H_i\mathcal{P}^\circ_i & = & \bigcup_{g \in H_i} g\mathcal{P}^\circ_i \\
 H_i\mathcal{P}_i & = &  \bigcup_{g \in H_i} g\mathcal{P}_i = \overline{H_i \mathcal{P}^\circ_i}.\\
 H_i\mathcal{P}^\star_i &=&  \bigcup_{g \in H_i} g\mathcal{P}^\star_i=  H_i\mathcal{P}_i \cap S_i^\star
\end{eqnarray*}
We call an edge $\tilde{e} \in \widetilde{\mathcal{E}}_i$ an \emph{interior edge} of $H_i\mathcal{P}_i$ 
if $\tilde{e} \subseteq g_1\mathcal{P}_i\cap g_2\mathcal{P}_i$ for distinct $g_1,g_2 \in H_i$.
The interior edges (except possibly the endpoints) lie in the interior of  $H_i \mathcal{P}_i$. 
The \emph{boundary edges} lie in $ \partial H_i \mathcal{P}_i$. 
Our next proposition gives a criterion for continuously extending $\bigcup_{g \in H} \phi_{g,g^\prime}$ to a bijective isometry
$H_1 \mathcal{P}_1^\star\rightarrow H_2 \mathcal{P}_2^\star$ or $H\mathcal{P}_1^\star\rightarrow H^\prime \mathcal{P}_2^\star$,
using the prime notation.

\begin{proposition}\label{prop-ctsext}
Let $S_1,S_2$ be two modular companions lying over their common quotient orbifold $T$ via the projections $\pi_i : S_i\rightarrow T$. 
Let $H$ be a subset of  $G$, containing the identity $1$ and all crossover transformations $\tau_{j,1}$.  
Let $w:H\rightarrow H^\prime \subseteq G$, $g\rightarrow w(g)=g^\prime$ be a $1-1$ assignment of the elements of  $H$ to $H^\prime\subseteq G$, satisfying:
\begin{eqnarray*}
w(1) &=& 1 \\
w(\tau_{j,1}) &=& \tau_{j,2} 
\end{eqnarray*} 
Then, the map   $\bigcup_{g \in H} \phi_{g,g^\prime}$ extends to a  bijective isometry 
$\phi^\star:H\mathcal{P}_1^\star \rightarrow H^\prime\mathcal{P}_2^\star$ 
if and only if for each interior edge $(e_j,g,g\tau_{j,1})$ of type $e_j$, and $g\in H$:
\begin{equation}\label{eq-MCcompatible}
w(g\tau_{j,1}) =  w(g)w(\tau_{j,1})=w(g)\tau_{j,2}.
\end{equation} 
\end{proposition} 

\medskip

\begin{proof}
We first note that if an edge of type $e_j$ of $g\mathcal{P}_1$  is in the interior of  $H\mathcal{P}_1$, 
then  $H$ must contain $g\tau_{j,1}$.
Now, suppose that  $g_1,g_2 \in H$ and  $g_1\mathcal{P}_1$ and  $g_2\mathcal{P}_1$ share an interior edge.
Then, $g_2=g_1\tau_{j,1}$ for some $j$.
If  $\phi^\star$ is continuous then  \eqref{eq-MCcompatible} must hold for $g=g_1$. 
Otherwise, there is a ``tearing'' discontinuity along $g_1\mathcal{P}_1 \cap  g_2\mathcal{P}_1$ 
since $w(g_1)\mathcal{P}_2$ and $w(g_2)\mathcal{P}_2$ do not meet along an interior edge of type $e_j$.  
This shows that the condition  \eqref{eq-MCcompatible} is necessary.

Next, suppose that there is an interior edge $\tilde{e} \subset H\mathcal{P}_1$ of type $e_j$.  
As above, there are $g_1,g_2 \in H$ such that $g_2=g_1\tau_{j,1}$ and $\tilde{e} \subseteq g_1\mathcal{P}_1\cap g_2\mathcal{P}_1$.  
There may be several lifts of distinct types in $g_1\mathcal{P}_1\cap g_2\mathcal{P}_1$ if there are repetitions of the value 
$\tau_{j,1}$ in the crossover sequence. 
Because $w(\tau_{j,1})=\tau_{j,2}$ and $w$ is injective, then  $\tau_{j,2} = \tau_{j^\prime,2}$ if and only if $\tau_{j,1} = \tau_{j^\prime,1}$, and we get the same repetitions in the crossover sequence for $S_2$.
Since
\[
w(g_2)= w(g_1\tau_{j,1})=w(g_1)w(\tau_{j,1})=w(g_1)\tau_{j,2},
\]
then $w(g_1)\mathcal{P}_2\cap w(g_2)\mathcal{P}_2$ contains edges of the same types as in   $g_1\mathcal{P}_1\cap g_2\mathcal{P}_1$.  

For each edge $\tilde{e}$ in $g_1\mathcal{P}_1\cap g_2\mathcal{P}_1$, construct a (simply connected)
tubular neighbourhood $U$ of the image $\pi_1(\tilde{e})$.
There are lifts $U_1$ and $U_2$ of $U$, respectively,  to tubular neighbourhoods of $\tilde{e}$ and its counterpart $\tilde{e}\,^\prime$ in  $w(g_1)\mathcal{P}_2\cap w(g_2)\mathcal{P}_2$. Using the ideas of Lemma \ref{lem-SClifts} and Lemma \ref{lem-doublelifts}, there is a specification of $\pi_2^{-1}\circ\pi_1: U_1 \rightarrow U_2$, which is a homeomorphism.  Because of the compatibility conditions, $\bigcup_{g \in H} \phi_{g,g^\prime}$  and 
$\pi_2^{-1}\circ\pi_1: U_1 \rightarrow U_2$ agree on their common domain $U_1-\tilde{e}$. By Lemma \ref{lem-extend2nipped} the map $\bigcup_{g \in H} \phi_{g,g^\prime}$ may be extended to include the interior of $\tilde{e}$. Do this for all edges in $g_1\mathcal{P}_1\cap g_2\mathcal{P}_1$.   
\end{proof}

\begin{remark}\label{rk-PIsoCayley} 
A partial isometry may be interpreted in terms of the embedded Cayley graphs. Let  $\mathrm{Cay}(H_1,E,\Sigma_1)$ 
be the subgraph of  $\mathrm{Cay}(G,E,\Sigma_1)$ whose vertex set is $H_1$ and contains all the edges of $\mathrm{Cay}(G,E,\Sigma_1)$
whose endpoints lie in $H_1$. Define a similar graph  $\mathrm{Cay}(H_2,E,\Sigma_2)$ for the second surface. 
Then the map in Proposition \ref{prop-ctsext} is a partial isometry if and only if the map
$\mathrm{Cay}(H_1,E,\Sigma_1) \rightarrow  \mathrm{Cay}(H_2,E,\Sigma_2)$ induced by $w:H_1\rightarrow H_2$ ($w:H\rightarrow H^\prime$)
is well defined and is a graph isomorphism, preserving edge types.
\end{remark}

\begin{remark}
If, as assumed in Proposition \ref{prop-ctsext}, 
\[
w : \{1\}\cup \{\tau_{j,1}: 1\le j\le 2k\}\rightarrow\{1\}\cup \{\tau_{j,2}: 1\le j\le 2k\}
\]
is bijective, then the edge collapse pattern and the multi-edge patterns must be the same. 
This is already a strong constraint on the crossover sequences, as discussed in Section \ref{subsec-modcompdiff}.    
\end{remark}

We finish the section with a corollary characterizing when modular companions are  conformally equivalent.

\begin{corollary}\label{cor-ctsext2}
Let all assumptions be as in Proposition \ref{prop-ctsext}. Then, a maximal partial isometry determined by $w: H\rightarrow H'$ 
is a complete isometry,  if and only if $H=G$ and  $w \in \mathrm{Aut}(G)$. 
In the case of a complete isometry the two modular companions are conformally equivalent by a conformal equivalence 
that  intertwines the group actions.  Furthermore, the Cayley graphs are isomorphic. The colouring of the edges can be implemented 
so that the colours are preserved by the isomorphism.
\end{corollary}

\begin{proof}
Most of the proof is straight forward except two small issues. If $H=G$ then $\phi^\star$ maps 
$S_1^\star$ isometrically to  $S_2^\star$, where $S_i^\star$ is the surface with all vertices of the tiling removed. 
By the Removable Singularity theorem, we can fill in all the punctures and still have a conformal map. 
The second issue is demonstrating that  $w \in \mathrm{Aut}(G)$. But, equation \eqref{eq-MCcompatible} holds
for all $g \in G$ and all crossover transformations $\tau_{i,1}$. 
Since the crossover transformations generate $G$, then $w$ must be a bijective homomorphism.    
\end{proof}

\subsection{Building partial isometries}\label{subsec-buildingpartialIso}

To make partial isometries meaningful and interesting we impose some requirements on $H$. 

\begin{enumerate}
  \item The set $H\mathcal{P}_1^\star$ is connected.
  \item The set $H$ is maximal, namely we cannot add an additional polygon on the boundary of $H\mathcal{P}_1^\star$ and still have a partial isometry. 
\end{enumerate}

Computational experiments using the Magma code in \cite{Br3} have shown that various maximal $H$ need not have the same number of elements, 
and so a maximal $H$ need not be unique. Of course, different selections of $\mathcal{E}$ lead to different 
shapes of polygons, different polygon boundary behaviour, and different results on maximality.

The rest of this section gives a detailed presentation of an algorithm for computing partial isometries 
satisfying the rules 1 and 2 above. 
Each invocation of the algorithm will start with a selection of $\mathcal{E}$, 
generating vectors $\mathcal{V}_1$ and $\mathcal{V}_2$, and an initial $H=\{1\}$. 
At each iteration of the algorithm a new $H$ is calculated, 
following the rules above for modifying $H$ or terminating the algorithm. 
After termination, the final $H$ and $|H|$ are reported. 
By allowing  $\mathcal{V}_1$ and $\mathcal{V}_2$ to vary, 
we can find maximal partial isometries between any two  modular companions.  
The algorithm even works for monodromy pairs $(T,\xi_1)$,  $(T,\xi_2)$ that are not modular companions. 
At the end of the section we present a toy example, computable by hand,  based on examples \ref{ex-S3-2233-1} and \ref{ex-S3-2233-2}.
This is followed by Section \ref{subsec-spectacular} giving a sampling of interesting, more complex examples.

The construction  starts with a seed polygon and then adds polygons one at a time satisfying 
the connectivity requirement and the compatibility equations  \eqref{eq-MCcompatible}.
We may assume that the seed polygon corresponds to the element $1 \in G$. To maintain connectivity and to extend the definition of $w$ we always add polygons that meet the boundary of the current $H\mathcal{P}_1$. 
The selection process for the new polygon can lead to different results for a maximal $H$. Here are two candidates for a selection process.
\begin{enumerate}
  \item \textbf{Cayley distance.} We have a nested structure $\{1\}=H_0 \subseteq H_1 \subseteq H_2 \subseteq \cdots = H$, 
  such that $H_n$ consists of all possible polygons that can be added and reached in $n$ steps or less. 
  The ``$n$ steps or less'' refers to paths in the dual Cayley graph starting at the vertex $1$. 
  The idea is that we first add all polygons $1$ step away, then $2$ steps away, ... , recognizing that some of the polygons cannot be added.
  \item \textbf{Random.} At each step, select a random polygon that meets the boundary of the current $H\mathcal{P}_1^\star$. 
  Repeated runs of this process will eventually sweep out all possible maximal $H$, though the number of different $H$ could be very large. 
\end{enumerate}

We are now going to try to extend $\phi_0$ by adding polygons one at a time to the seed polygon $\mathcal{P}_1$ and determining the values of $w(g)$ so that $\bigcup_g \phi_{g,g^\prime}$ remains continuous when we add the polygons to the domain. We will record our additions by constructing corresponding subgraphs 
$\mathrm{Cay}(H,E,\Sigma_1)$ and its isomorphic image in $\mathrm{Cay}(H^\prime,E,\Sigma_2)$.

We will need several sets and lists to keep track of our partial map domain $H\mathcal{P}_1^\star$  as we complete the steps of our algorithm.
\begin{itemize}
  \item $H$:  List of polygon labels for $S_1$ in the current domain. 
  \item $w = \left\{ (h,w(h))=(h,h^\prime) \,|\, h \in H\right\}$: Assignment of polygon labels in the partial map $\overline{\phi}:H\mathcal{P}_1^\star \rightarrow H^\prime\mathcal{P}_2^\star$. 
  \item $wH = \left\{ w(h)\,|\, h \in H\right\}=H^\prime$:  List of polygon labels for $S_2$ using the current domain.     
  \item $IntEdges_1$:  List of current interior edges in $H\mathcal{P}_1$. This is the edge set of the current  $\mathrm{Cay}(H,E,\Sigma_1)$.
  \item $IntEdges_2$:  List of current interior edges in $H\mathcal{P}_2$. This is the edge set of the current  $\mathrm{Cay}(H^\prime,E,\Sigma_2)$.
  \item $\partial Edges_1$: List of current boundary edges in  $H\mathcal{P}_1$.
  \item $BadEdges_1$:  List of boundary edges for which a new polygon could not be added.
  \item $op$: the induced map on the edges of  $\mathrm{Cay}(G,E,\Sigma_1)$: $(e, g, g\tau_e)\rightarrow (e^{op}, g\tau_e,g)$
\end{itemize}

\begin{remark}
  Though we use set notation, the above quantities are lists and are treated as such in the computer calculations. New items, such as the current boundary are always added at the end. Deletions are done in place and otherwise do not alter the order of the list. Thus the order in which items are added or processed is recorded.  

\end{remark}

\begin{algorithm}
After initialization, we iteratively repeat the remaining steps below until there are no good boundary edges left.

\begin{enumerate}
  \item Initialize: $H = \left\{ 1 \right\}$;  $w = \left\{(1,1)\right\}$; $IntEdges_1$, $IntEdges_2$, and $BadEdges_1$ are empty,
  $\partial Edges_1 =  \left\{ (e_1, 1, \tau_{1,1}),\ldots,(e_{2k}, 1, \tau_{2k,1})\right\}$.
  \item Select a boundary edge $(e, g, g\tau_e)$ in  $\partial Edges_1$.  For the Cayley distance method, select the first edge in $\partial Edges_1$.  For the  random method select a random edge in $\partial Edges_1$. 
  \item With the selected edge do the following:  
  \begin{enumerate}
    \item Set $h=g\tau_e$ and assume $e=e_j$ so that $\tau_e= \tau_{j,1}$. Compute the test value $w(h)=w(g)w(\tau_{i,1}) =w(g)\tau_{i,2}$. 
        
    \item Check to see if the test value $w(h)$ keeps $w$ injective,  i.e., $w(h) \notin wH$ should hold.
    \item Determine all the edges $(e^\prime,g,g\tau_{e^\prime})$ such that $g\tau_{e^\prime} = h$ and $g \in H$, 
          as these edges will become interior if we add $h$ to $H$.
    \item For each such edge, test to see if the continuity criterion holds \eqref{eq-MCcompatible},
          i.e.,  $w(h)=w(g\tau_{e^\prime})=w(g)w(\tau_{e^\prime})$, is satisfied.
    \item If all tests are passed, then do the following:  (1) add $h$ to $H$ and update $wH$, 
          (2) add the new interior edges to $IntEdges_1$ and $IntEdges_2$ at the end of the lists, 
          (3) remove the new interior edges from  $\partial Edges_1$, and 
          (4) append all the remaining edges of $\partial h\mathcal{P}_1$ to $\partial Edges_1$,  at the end of the list.
    \item If any of the tests fail then do this: (1) remove all the edges of 
          $\partial h\mathcal{P}_1$ from  $\partial Edges_1$, maintaining list order  and 
          (2)  append all the edges of $\partial h\mathcal{P}_1$ to  $BadEdges_1$, at the end of the list:  
  \end{enumerate}
  \item If $\partial Edges_1$ is non-empty the return to Step 2, otherwise halt.     
\end{enumerate}
\end{algorithm}

\begin{example}\label{ex-S3-2233-4}
The number of tiles in  a partial isometry is a measure of how closely two modular companions resemble each other geometrically, 
i.e., the size of a ``patch match''.  Since the tiles all have the same area the fraction $|H|/|G|$ 
is the proportion of the total area of the surface contained in a maximal patch. 
The size of $|H|/|G|$ will vary according to the cut system chosen, the selection process for adding tiles, 
and perhaps even the starting point for ordering the edges of a polygon.

The partial isometry information can be put into matrix-like table, where the entry
is the number of tiles in a maximal partial isometry from the surface  defined by the row label to the surface defined as the column label.
The row and column labels are the representative generating vectors defined in Example \ref{ex-S3-2233-2}.
Observe that the diagonal elements equal $|G|$ and that the off diagonal elements are strictly smaller. 

\[%
\label{table-4}
\begin{tabular}
[c]{c}%
\begin{tabular}
[c]{|l|l|l|}\hline
          & $\mathcal{V}_1$ & $\mathcal{V}_2$\\ \hline
$\mathcal{V}_1$ & $6$ & $2$ \\\hline
$\mathcal{V}_2$ & $1$ & $6$ \\\hline
\end{tabular}
\ \\
\\
Table 4. Partial isometry matrix for $G  = \Sigma_3$  and $\mathfrak{s}=(0;2,2,3,3)$
\end{tabular}
\
\] 

\end{example}

\subsection{Finding interesting examples}\label{subsec-spectacular}
The toy example we have been considering is a bit unsatisfying, so we sketch how we may construct more interesting 
examples. Given a group signature pair $(G,\mathfrak{s})$, there are two steps to analyzing 
the corresponding  modular companions $S$ lying over a typical quotient $T$. 
First, we identify all the surfaces lying over  $T$ by means of generating 
vector classes   $\mathcal{V}^{\mathrm{Aut}(G)}$ and then we partition  
the vector classes into subsets corresponding to 
modular companions. Each subset consists of those surfaces lying over $T$ that 
lie within a single stratum of the $(G,\mathfrak{s})$ action moduli space. The subsets also correspond
to modular orbits of $M_\mathfrak{s}$ acting upon the vector classes via equation \eqref{eq-GVmodaction}.
The second step selects a cutset $\mathcal{E}$, converts generating vectors to crossover sequences, 
determines any tiling deficiencies, and then computes a partial isometry matrix. 

The theory and Magma code for the first step were developed in \cite{BrCoIz1} and \cite{BrCoIz2} 
and the code and log files have been posted on \cite{Br3} (Sections 1 and 2). The code runs in reasonable time for groups 
with orders in the low hundreds. Several thousand cases were computed in \cite{BrCoIz2}, 
corresponding to groups of low order in the Magma small group data base and such that $T$ 
is a torus with a single cone point.  For the current paper the code from \cite{BrCoIz1} and \cite{BrCoIz2}
was modified and extended to include both Steps 1 and 2, and the updated code
and logfiles have been posted in Section 4 of \cite{Br3}.

We have described Step 2 in great detail in this paper, but, for completeness, 
we will say few words about Step 1 (full details are in  \cite{BrCoIz1}). 
First, we compute all generating vectors with the given signature. This is the 
computationally expensive step, though there are tricks to make the computation go faster. 
Next, we need a procedure to find unique representatives of the automorphism classes 
$\mathcal{V}^{\mathrm{Aut}(G)}$. Pick a desirable, arbitrary 
but fixed, ordering of the elements $G$, this allows us to lexicographically sort generating vectors. 
For any class  $\mathcal{V}^{\mathrm{Aut}(G)}$, let  $\overline{\mathcal{V}}$ 
denote the (unique) lexicographic minimum of all the vectors in $\mathcal{V}^{\mathrm{Aut}(G)}$. 
We then have a sequence $\overline{\mathcal{V}_1},\ldots,\overline{\mathcal{V}_N}$ of representatives 
of the automorphism classes, where $N$ is the number of surfaces lying over a generic $T$.

In the planar case, the modular group $M_\mathfrak{s}$ has the following structure. 
There is a normal subgroup ${PM}_\mathfrak{s}$, corresponding to the pure braid group, 
derived from homeomorphisms of $T$ that do not move cone points.
The full group $M_\mathfrak{s}$ is generated by adding to ${PM}_\mathfrak{s}$  those $\Phi_{j,j+1}$ 
in Example \ref{ex-S3-2233-2} that preserve the signature. The quotient $M_\mathfrak{s}/{PM}_\mathfrak{s}$
is the permutation group of the cone points that preserve cone order. 
Now, a set of generators for ${PM}_\mathfrak{s}$ is well known, so we may easily construct a finite set of generators
for $M_\mathfrak{s}$ by adding the appropriate automorphisms $\Phi_{j,j+1}$.  
For the prior statements on structure to hold, the signature is assumed be in non-decreasing order. 

For any modular element $\mu \in M_\mathfrak{s}$, a permutation $q_\mu$  in the symmetric group 
$\Sigma_N$ is determined by 
\begin{equation}\label{eq-mod2perm}
  q_\mu: \overline{\mathcal{V}_i}\rightarrow  \overline{\mathcal{V}_{i^\prime}} = \overline{\mu\cdot \mathcal{V}_i},
\end{equation}
where  $\mu\cdot \mathcal{V}_i$ is as defined in \eqref{eq-GVmodaction}. 
Following Example \ref{ex-S3-2233-2}, we construct a finite permutation group
$Q_\mathfrak{s}=\langle q_\mu\rangle$ where $\mu$ runs over a known finite 
set of generators of $M_\mathfrak{s}$, described above. We now ask Magma to compute 
the orbits of $Q_\mathfrak{s}$ acting on 
$\overline{\mathcal{V}_1},\ldots,\overline{\mathcal{V}_N}$.
The orbits correspond to strata and the vectors within an orbit correspond 
to modular companions within a stratum.  Typically, $Q_\mathfrak{s}$ is rather large.

\[\label{table-7}
\begin{tabular}[c]{c}
\begin{tabular}
[c]{|c|c|c|c|c|c|}
\hline
$G$       & $|G|$ & $\mathfrak{s}$ & $\sigma$ & \#surfs & \#ModComp  \\
\hline
$Sym(3)$  & $6$   & $(2,2,3,3)$                & $2$              &  $2$    &  $(2)$           \\ 
\hline
$Cyclic(13)$  & $13$  & $(13,13,13,13)$        & $12$              &  $133$    & $(3, 4, 6, 12, 12, 12, 12, 12, 12, 24, 24 )$           \\ 
\hline
$SG(21,1)$  & $21$  & $(3,3,7,7)$        & $12$              &  $12$    & $(6, 6)$           \\ 
\hline
$Alt(5)$  & $60$  & $(2,2,2,3)$                 & $6$              &  $9$    & $(9)$           \\ 
\hline
$Alt(5)$  & $60$  & $(2,3,3,5)$                & $20$              &  $20$    & $(20)$          \\ 
\hline
$Alt(5)$  & $60$  & $(5,5,5,5)$                & $37$              &  $47$    & $(6,10,15,16)$            \\ 
\hline
$PSL(2,7)$  & $168$  & $(2,2,3,3)$             & $29$              & $15$    &  $(15)$           \\ 
\hline
$PSL(2,7)$  & $168$  & $(7,7,7,7)$             & $121$              & $95$    &  $(6,7,16,24,42)$            \\ 
\hline
$PSL(2,11)$  & $660$  & $(5,5,5,5)$            & $397$              & $4906$    & not done            \\ 
\hline
\end{tabular}
\ \\
\\
Table 5. Strata and modular companion information for selected pairs $(G,\mathfrak{s})$ 
\end{tabular}
\]

We finish by citing two families of $(G,\mathfrak{s})$ actions from previous work (\cite{BrCoIz1},\cite{BrCoIz2};
discussing a list of examples in Table 5, computed specifically for this paper; and then present a partial isometry matrix 
in Table 6.  In Table 5, above, we present information for the covering map 
$p: \mathcal{S}_{G,\mathfrak{s}}\rightarrow  \mathcal{M}_\mathfrak{s}$ \eqref{eq-pBM}.   
The first two columns are group information, the 3rd column is the signature, the 4th column is the genus of $S$, 
and the 5th column is total number of surfaces lying over a generic $T$. 
The 6th column is a list of numbers of modular companions when restricted to a component of the cover, 
see equations \eqref{eq-modcomp-comp} and \eqref{eq-modcomp-map}. The numbers are also the orbit sizes for the $Q_\mathfrak{s}$ action.

\paragraph{\emph{Prime order actions.}} The strata corresponding to actions of prime cyclic groups 
of order $p$ are ubiquitous in moduli space.  In our case, the quotient surface is a sphere 
with four branch points of order $p$ and the surfaces lying over have genus $p-1$. 
In \cite{BrCoIz1} we studied the growth of the number of strata for primes $\le 101$. 
A chart showing the growth of the number of strata as $p$ increases is given in the paper
and is posted on the site \cite{Br3}. We have included one line of the chart in Table 5, for the prime $13$, 
and already the number of strata is large.  

\paragraph{\emph{Non-abelian groups of order $pq$.}} A second interesting family consists of the non-abelian groups of order $pq$ 
where $p,q$ are odd primes with $p\vert(q-1)$. We studied this family in \cite{BrCoIz2}
because the family produced surface actions for which the quotient is a torus with one cone point.
The family was amenable to ``by hand analysis" in \cite{BrCoIz2}. The signature $(0;p,p,q,q)$ has numerous 
generating vectors yielding surfaces of genus $(p-1)(q-1)$. However we do not attempt to do 
a similar  ``by hand analysis" for these signatures in this paper. The smallest example is $SG(21,1)$ 
the first small group of order 21 in the small group database. The information for
$SG(21,1)$ is given in the Table 5. 

\paragraph{\emph{Table 5 special cases.}}The last six lines of Table 5 explore small simple group actions with signatures that have symmetries.
The table also shows the explosive growth of the number of modular companions as the group size increases.
Detailed, auto-generated  log files, showing all the steps and results of the calculations are posted on \cite{Br3}. 

\paragraph{\emph{Partial isometry matrix.}} Finally, we present in Table 6, the partial isometry  matrix for $G  = \mathrm{Alt}(5)$  and $\mathfrak{s}=(0;2,2,2,3)$. 
We used the cut system $\mathcal{E}_4$ and the Cayley distance criterion for the addition of new tiles. 
All the off diagonal entries are quite a bit less that $|G|$, showing that the off diagonal pairs are not conformally equivalent.
The diagonal entries must be  $|G|$ according to Corollary \ref{cor-ctsext2}. 
We have not analyzed any patterns in the partial isometry matrix.

\[%
\label{table-8}
\begin{tabular}
[c]{c}%
\begin{tabular}
[c]{|l|l|l|l|l|l|l|l|l|l|}\hline
          & $\mathcal{V}_1$ & $\mathcal{V}_2$ & $\mathcal{V}_3$ & $\mathcal{V}_4$ & $\mathcal{V}_5$ 
          & $\mathcal{V}_6$ & $\mathcal{V}_7$ & $\mathcal{V}_8$ & $\mathcal{V}_9$\\ \hline
$\mathcal{V}_1$ & $60$ & $36$ & $36$ & $35$ & $35$ & $35$ & $36$ & $36$ & $36$\\\hline
$\mathcal{V}_2$ & $37$ & $60$ & $37$ & $39$ & $42$ & $37$ & $37$ & $42$ & $39$\\\hline
$\mathcal{V}_3$ & $35$ & $35$ & $60$ & $44$ & $44$ & $35$ & $35$ & $35$ & $35$\\\hline
$\mathcal{V}_4$ & $36$ & $36$ & $42$ & $60$ & $42$ & $36$ & $36$ & $36$ & $36$\\\hline
$\mathcal{V}_5$ & $33$ & $44$ & $44$ & $44$ & $60$ & $33$ & $33$ & $44$ & $33$\\\hline
$\mathcal{V}_6$ & $37$ & $34$ & $37$ & $37$ & $34$ & $60$ & $34$ & $34$ & $37$\\\hline
$\mathcal{V}_7$ & $38$ & $38$ & $36$ & $36$ & $36$ & $36$ & $60$ & $38$ & $38$\\\hline
$\mathcal{V}_8$ & $38$ & $44$ & $38$ & $38$ & $44$ & $38$ & $38$ & $60$ & $38$\\\hline
$\mathcal{V}_9$ & $35$ & $39$ & $35$ & $39$ & $35$ & $35$ & $35$ & $35$ & $60$\\\hline
\end{tabular}
\ \\
\\
Table 6. Partial isometry matrix for $G  = \mathrm{Alt}(5)$  and $\mathfrak{s}=(0;2,2,2,3)$
\end{tabular}
\
\]

\section{Basic Lemmas}\label{sec-BasicLemmas}
In the preceding sections we have used some basic facts, collected in this section as lemmas.
The lemmas are very basic and so we leave the proofs to the reader.  

\begin{lemma}\label{lem-SClifts}
Let $p:X \rightarrow B$ be an unbranched covering space and let $A \subset B$ be a locally closed, path connected, and simply connected subset of  $B$. Then $p^{-1}(A)$ is a locally closed subset of $X$ such that each path component of  $p^{-1}(A)$ is mapped by $p$ homeomorphically onto $A$. If $A$ is an open subset and $p$ is smooth or holomorphic then each path component of $p^{-1}(A)$ is mapped diffeomorphically or biholomorphically, respectively, onto $A$.      
\end{lemma} 

\begin{lemma}\label{lem-doublelifts}
For $i=1,2$ let $p_i:X_i \rightarrow B$ be unbranched covering spaces, let $b_0 \in B$, 
and  select  $x_j \in X_j$ that lie over $b_0$. Then, there is a small neighbourhood 
$U_0$ of $b_0 \in B$, and neigbourhoods $U_j \subseteq p_j^{-1}(U_0)$ of  $x_j \in X_j$ 
such that $p_j:U_j \rightarrow U_0$ is a homeomorphism. 
The restricted map $\psi_{x_1,x_2}=p_2\circ p_1^{-1}:U_1\rightarrow U_2$ is also a homeomorphism. 
If the covering space maps are smooth or holomorphic then so is the restricted map is a 
diffeomorphism or biholomorphism, respectively.
\end{lemma} 

The next lemma is restricted to sequential topological spaces so that sequences may be used in its proof.

\begin{lemma}\label{lem-extension}
 Suppose that $X$ and $Y$ are sequential topological spaces, $U\subset X$ and let $f:U\rightarrow Y$ 
be any continuous map.  Let $W$ be a subset satisfying $U\subseteq W \subseteq \overline{U}$, 
and such that for every $x \in W$ there are a neighbourhood $V_x$ of  $x$ in $X$ 
and a continuous map  $f_x:V_x \rightarrow Y$ such that $f=f_x$ when restricted to $U\cap V_x$. 
Then $f$ has a unique continuous extension $f:W\rightarrow Y$.    
\end{lemma}


\begin{thebibliography}{99}
\bibitem{BaCoIz} G. Bartolini A.F. Costa, M, Izquierdo:\textit{ On the orbifold structure of the moduli space of Riemann
surfaces of genera four and five}, Rev. R. Acad. Cienc. Exactas Fis. Nat. Ser. A Mat. RACSAM \textbf{108} (2014), no. 2, 769–793.
\bibitem{Br1} S.A. Broughton, \textit{The equisymmetric stratification of the moduli space and the Krull dimension of mapping class groups}, Topology Appl. \textbf{37} (1990) 101-113.
\bibitem{Br2} S.A. Broughton, \textit{ Equivalence of finite group actions on Riemann surfaces and algebraic curves}, Automorphisms of Riemann surfaces, subgroups of mapping class groups and related topics, 89--132, Contemp. Math., 776, Amer. Math. Soc., [Providence], RI, [2022], c 2022.
\bibitem{Br3} S.A. Broughton \textit{Tilings, Geometry, and Automorphisms of Surfaces}, \url{https://tilings.org/autosurf}
\bibitem{BrCoIz1} S.A. Broughton, A.F. Costa, M. Izquierdo, \textit{One dimensional equisymmetric strata in moduli space}, Automorphisms of Riemann surfaces, subgroups of mapping class groups and related topics, 177--215, Contemp. Math., 776, Amer. Math. Soc., [Providence], RI, [2022], c 2022.
\bibitem{BrCoIz2}S.A. Broughton, A.F. Costa, M. Izquierdo  \textit{One Dimensional Equisymmetric Strata in Moduli Space with Genus 1 Quotient Surfaces}, 
 Rev. Real Acad. Cienc. Exactas Fis. Nat. Ser. A-Mat. (2024) 118:21
\url{https://doi.org/10.1007/s13398-023-01520-9}
\bibitem{BrPaWo}S.A. Broughton, J. Paulhus, A. Wooton, \textit{ Future directions in automorphisms of surfaces, graphs, and other related topics } Automorphisms of Riemann surfaces, subgroups of mapping class groups and related topics, 37--67, Contemp. Math., 776, Amer. Math. Soc., [Providence], RI, [2022], c 2022.
\bibitem{GDH}G. Gonz\'{a}les D\'{i}ez, W.J. Harvey, \textit{ Moduli of Riemann surfaces with symmetry}, 
in: Harvey WJ, Maclachlan C, eds. Discrete Groups and Geometry. London Mathematical Society Lecture Note Series. Cambridge University Press; 1992:75-93.
\bibitem{Ha} W. Harvey, \textit{On branch loci in Teichmüller space}, Trans. Amer. Math. Soc. \textbf{153}, 387–399 (1971).
\bibitem{HPCRH} R.A. Hidalgo, J. Paulhus, S. Reyes-Carocca, A.M. Rojas, \textit{On non-normal subvarieties of the moduli space of Riemann surfaces}, 
 	\url{https://arxiv.org/abs/2306.01673} 
\bibitem{JS}G.A. Jones and D. Singerman \textit{Complex Functions} Cambridge University. Press, Cambridge, 1988.
\bibitem{MaSi} Macbeath, A.M., Singerman, D., \textit{Spaces of subgroups and Teichmüller space},Proc. London Math. Soc.
\textbf{531} (3) 31 (1975), 2, 211–256
\bibitem{Mag}Magma computational algebra system, Computational Algebra Group, University of Sydney.
\bibitem{Mask}B. Maskit, \textit{On Poincare’s theorem for fundamental polygons}, Adv. in Math., \textbf{7} , 219-230  (1971).
\bibitem{Mas}W.S. Massey,  \textit{ A Basic Course in Algebraic Topology}, GTM \textbf{127}, Springer, New York (1991).
\bibitem{MM} Y. Matsumoto and J. M. Montesinos,\textit{ A proof of Thurston’s Uniformization Theorem of Geometric Orbifolds}. Tokyo J. Math, 14 no. 2 (1991)
\bibitem{Si1} D. Singerman. \textit{Subgroups of Fuschian groups and finite permutation groups}, Bull. Lond. Math. Soc. 2, \textbf{535} 319–323 (1970)
\bibitem{Si2} D. Singerman \textit{Automorphisms of maps, permutation groups and Riemann surfaces}, Bull Lond. Math. Soc.  \textbf{8}  65-68 (1976).
\bibitem{ZVC} H. Zieschang, E. Vogt, and H.D. Coldewey, \textit{Surfaces and planar discontinuous groups}, Lecture Notes in Mathematics, \textbf{835}, Springer, (2006).
\end{thebibliography}
\end{document}